\documentclass{elsarticle}
\usepackage{amssymb}
\usepackage{amsmath}
\usepackage{amsfonts}
\usepackage{amsbsy}
\usepackage{amscd}
\usepackage{amsthm}
\usepackage{times}
\usepackage{psfrag}
\usepackage{graphics}
\usepackage{color}
\usepackage{colortbl}
\usepackage{graphicx}
\usepackage{caption}
\usepackage{subcaption}
\usepackage{array,supertabular}
\usepackage{tabularx}
\usepackage{booktabs}
\usepackage{multirow}
\usepackage{ulem}
\usepackage{units}
\usepackage{mathabx}
\usepackage{accents}
\usepackage{algorithmicx}
\usepackage{algorithm}
\usepackage{xfrac}
\usepackage{algpseudocode}
\usepackage{rotating}
\usepackage[capitalize]{cleveref}

\crefname{secinapp}{Section}{Sections}
\Crefname{secinapp}{Section}{Sections}

\newlength{\dhatheight}

\newtheorem{theorem}{Theorem}[section]

\newtheorem{lemma}[theorem]{Lemma}
\newtheorem{proposition}[theorem]{Proposition}

  \newcommand{\Ex}{\mathbb{E}}
\newcommand{\p}{\mathbb{P}}

\def\u{{\bf u}}

\begin{document}
\begin{frontmatter}
\title{A continuous-time model of
  centrally 
  coordinated
  motion with random switching 
  }

\author[byu1]{J. C. Dallon}

\author[byu1]{Lynnae C. Despain}

\author[byu1]{Emily J. Evans\corref{cor1}}
\ead{ejevans@math.byu.edu}

\author[byu1]{Christopher P. Grant}

\author[byu1]{W. V. Smith}

\cortext[cor1]{Corresponding author}

\address[byu1]{Department of Mathematics,
  Brigham Young University,
  Provo, Utah 84602, USA}

\begin{abstract}
This paper considers differential problems with random switching, with specific applications to the motion of cells and centrally 
coordinated
motion.  
Starting with a differential-equation model of cell motion that was proposed previously, we set the relaxation time to zero and consider the simpler model that results.  
We prove that this model is well-posed, in the sense that it corresponds to a pure jump-type continuous-time Markov process (without explosion).
We then describe the model's long-time behavior, first by specifying an attracting steady-state distribution for a projection of the model, then by examining the expected location of the cell center when the initial data is compatible with that steady-state.  
Under such conditions, we present a formula for the expected velocity and give a rigorous proof of that formula's validity.
We conclude the paper with a comparison between these theoretical results and the results of numerical simulations.
\end{abstract}


\end{frontmatter}

\section{Introduction}
This paper studies the motion of cells, and at the same time certain larger questions surrounding differential problems with random switching. Cell motion is fundamental in many systems
including wound healing \cite{Krawczyk:1971:PEC,Tanner:2009:CMC}, cancer \cite{Yilmaz:2010:MMM}, and morphogenesis \cite{Keller:2000:MCE,Mammoto:2010:MCT,Rieu:2009:MDS}. 
 In \cite{Dallon:2013:FBM} we introduced a differential equation model for cell motion which suggests that cell speed is independent of force. A possible explanation for that predicted behavior is provided by the saltatory nature of cell adhesion and thus motion, an idea also suggested by numerical simulations. Moreover, data suggest that there is little or no functional dependence of speed on cell force \cite{Dallon:2013:CSI}. Typically, biologists determine cell speed by averaging cell displacement measured on the order of minutes. The implication is that this type of average is largely independent of cell force and highly dependent on the adhesion dynamics.
 
 The present model is force based and focuses on the random nature of integrin based adhesion sites. We do not model the molecular processes involved in adhesion. Rather, our purpose is to link the statistical properties of the dynamics of the adhesion process to overall cell motion. Connecting the two types of models is a challenging goal for the future. 
The insight and the heart of the model is this conjecture that speed is highly dependent on adhesion dynamics. To better understand the model, our eventual aim is to rigorously prove the conjecture. In order to do this, we simplified the model in two steps. First, a centroid model was devised which is a limiting case of the differential equation model as the forces get large. The next simplification was to consider a discrete-time centroid model. Our analysis started with this discrete-time model in \cite{Dallon:2013:CSI}. This discrete-time centroid model was analyzed using Markov chain theory. In this paper we add time back to the model returning to the first simplification. Here, we construct a complete continuous time centroid model that parallels the differential-equation model in the sense noted. For convenience, we refer to this new model as the CTCM, for Continuous-Time Centroid Model.

While the CTCM no longer involves differential equations, we prove that it is well-posed; \textit{i.e.}, mathematically coherent.  
In particular, we prove that it corresponds to a pure jump-type continuous-time Markov process with a general (uncountable) state space.    
In order to demonstrate this, it is necessary to carefully construct (and validate) an appropriate transition kernel that encapsulates the instantaneous evolution law.
In the terminology of Kallenberg \cite{Kallenberg}, this is a \textit{rate kernel}, which is most easily constructible as the product of a probability kernel, known as the \textit{jump transition kernel}, and a real-valued function, known as the \textit{rate function}, defined on the state space.

With the CTCM on firm mathematical footing, we then proceed to rigorously analyze its long-time behavior.  
It is anticipated that the CTCM itself will not approach a limiting configuration (even probabilistically) because of the potential for cells to drift.  
Because of this (and the complexity) of the CTCM, we project its state space onto a finite set and examine the finite-state Markov process that is induced by this projection.  
A calculation will produce a steady-state distribution for this finite-state process, general theory will show that this distribution is attracting, and a version of Dynkin's criterion will establish the existence of ``pullback'' distributions for the CTCM that exhibit a sort of partial invariance.

The configuration for the CTCM after sufficient time has elapsed should be well-approximated by one of these pullback distributions.
Thus, to capture typical long-time behavior, we run the CTCM with such a pullback distribution as initial data.
We will prove that when we do so the expected location of the cell center is a continuous function of time.  
A formula for the expected velocity of the center can be written down, and we give a rigorous proof of the correctness of this formula.
This formula will be vital to proving the primacy of adhesion dynamics in cell motion in the full differential equation model.


We begin in Section~\ref{sec:demodel} by reviewing the differential equation model introduced in~\cite{Dallon:2013:FBM}.
The definition, justification, and well-posedness of the CTCM follow in Section~\ref{sec:wellposedness}.
In Section~\ref{sec:ltbehavior}, we analyze the long-time behavior of the CTCM.
%
%
In Section~\ref{sec:num}, we compare numerical results for the expected velocity of the CTCM to the theoretical results given by Theorem 2. In this section we also show numerical results for wait time distributions that are not exponential.  In Section~\ref{sec:discussion}, we conclude with a discussion that summarizes our results and suggests future applications of our work.

\section{Differential Equation Model}\label{sec:demodel}
The cell is modeled as a nucleus and multiple interaction sites which
exert forces on the nucleus as shown in Figure~\ref{fig:cell}.  These
interaction sites are integrin based adhesion sites (I-sites)
\cite{Friedl:2009:CCM,Ulrich:2009:TCM,Gumbiner:1996:CAM}.  I-sites
attach to an external substrate and once attached remain fixed to that
substrate location.  The duration of the attachment is determined by a
given probability distribution.  The same is true for the time the I-site
remains unattached, although the distributions need not be the same.
%
The model assumes the I-sites exert forces on the nucleus according to
Hooke's law; that is, the force is proportional to distance.  Let $\alpha_i$ denote the spring constant for the $i$th adhesion site.  Thus it is
as if the I-sites are attached to the cell center with springs which
have a rest length assumed to be zero.  Moreover there is a drag force on the
cell nucleus which is modeled assuming the center (nucleus) is a
sphere in a liquid with low Reynolds number and is proportional to
the velocity, denoted by $C$.  Denote the
location of the cell center as ${\bf x}$, a point in $\mathbb{R}^N$.
Likewise the location of each I-site $\u_{i}$ are points in
$\mathbb{R}^N$, where $i$ ranges from 1 to $n$. The random variable $\psi_i$ indicates whether the $i$th I-site is attached or detached.    Due to low Reynolds
number the acceleration term can be ignored and the equations of
motion are first order \cite{Dallon:2004:HCM}.  These equations are
\begin{equation}
C {\bf x}^{\prime} =\sum_{i=1}^n
-\alpha_i({\bf
  x}-\u_i) 
\psi_i(t),
\label{equ:ode}
\end{equation}
where $\u_i$ is given by 
\begin{equation}
\label{equ:smi}
\u_i(t) = \lim_{y \nearrow a_{p,i}}{\bf x}(y)+{\bf b}^{p,i}\mbox{ for } a_{p,i} \leq
t<a_{p+1,i} ,
\end{equation}
for each $i$ the sequence $\{a_{p,i}\}$ of random variables are the times when $\psi_i$ makes the
transition from 0 to 1, and $\{d_{p,i}\}$ is the sequence of random variables of the times when $\psi_i$ makes the
transition from 1 to 0. Of course, the two sequences are not independent since $a_{p,i}<d_{p^*,i}<a_{p+1,i}$ where $p^*=p$ if the initial state starts with the $i$th I-site attached and $p^*=p+1$ if it starts out detached. The vectors ${\bf b}^{p,i}$ are independent, identically distributed random
vectors with a distribution $\eta$. 
Although the equations of motion are independent of the location of
the I-site when it is detached, for convenience we define the location
to remain the same until it reattaches.
\begin{figure}[h]
  \centerline{\includegraphics*[width=.29\textwidth]{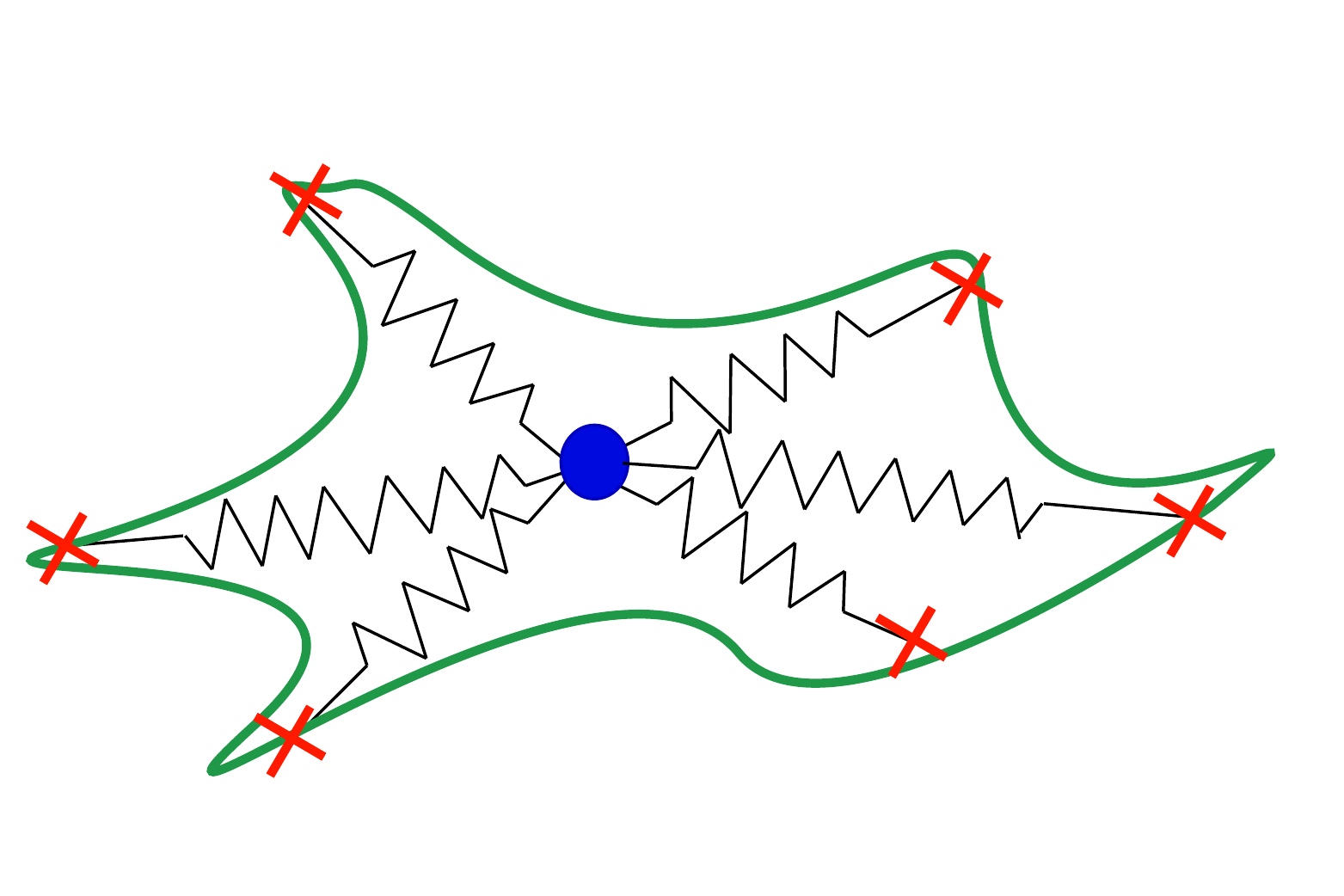}}
  \caption{This figure depicts the way a cell is modeled
    mathematically.  The cell is a center location (nucleus) with
    attached springs.  The other end of the springs are attached to
    I-sites which can interact with the substrate depicted
    by ``x''.}
\label{fig:cell}
\end{figure} 

\section{The CTCM and its Well-posedness}\label{sec:wellposedness}

Consider the model obtained when the following symmetrizing and limiting transformations are performed on \eqref{equ:ode} and \eqref{equ:smi}:
\begin{itemize}
 \item the $\alpha_i$ are considered to be independent of $i$;
 \item $C$ is set to $0$;
 \item The interevent times for detachment, $d_{p^*,i}-a_{p,i}$, are taken to be independent, identically-distributed exponential random variables;
 \item The interevent times for attachment, $a_{p+1,i}-d_{p^*,i}$, are taken to be independent, identically-distributed exponential random variables.
\end{itemize}
We no longer have a system of differential equations but still seem to have a sensible evolution law.  We can abstract from the particular context of biological cell motion, to describe the situation from scratch as follows.
We have finitely many objects that move through a physical space.
At any given time, in addition to having a location, each object has a ``status'', either ``attached'' or ``detached''.  
At random, and independently of one another, the objects change status.  All attached objects detach at the same rate, 
and all detached objects attach at the same rate (possibly different than the detachment rate).  The expected wait time
for a given object to change status depends only on its current status, not on how long it has had that status.
Any given object only changes location at the moment it attaches.  Its new location is a random perturbation from the 
centroid of the locations of the objects that were attached just before the given object attached.  
Whenever the only attached object detaches (leaving no objects attached), the ``centroid'' is considered to 
be the same as it was immediately before this detachment occurred.

To make sure that this described model is mathematically coherent, we will formally situate it in the theory of Markov processes on general state spaces.  Doing so will also allow us (in Section \ref{sec:ltbehavior}) to make use of the theorems of that theory to obtain rigorous results about centroid motion, rather than settling for heuristics.

As we go through the somewhat lengthy process of \textit{defining} the mathematical model, we will simultaneously \textit{interpret} its various formulas in order to argue that it really matches the informal description given above.
 
\subsection{General Notation}\label{notation}


In order to keep abuse of notation to a minimum, we distinguish between different types of Cartesian products, including one that is slightly more general than is typically used. 

Given two sets $B$ and $D$, we consider the product $B \times D$ to be a set of ordered pairs, but
we consider product of more than two sets to have functions as their elements.  (The motive for this is that it's easier to splice together explicitly indexed functions than positionally indexed tuples.)
Let $B$ be a set, $D$ be a finite set, and $P$ be a partition of $D$.  As usual, we define $B^D$ to be the set of functions from $D$ to $B$.  Thinking of functions as sets of ordered pairs, we note that if $f_p$ is an element of the set $B^p$ of functions from $p$ to $B$ for every $p \in P$, then $\bigcup_{p \in P} f_p \in B^D$.  If, for each $p \in P$, $B_p \subseteq B^p$, then we define the product set
\[
 \bigtimes_{p \in P} B_p := \left\{\bigcup_{p \in P} f_p : \text{$f_p \in B_p$ for every $p \in P$}\right\}.
\]

Suppose additionally that $B$ is a topological space.  If, for each $p \in P$, $\lambda_p$ is a Borel measure on $B^p$, then we define the product measure $\bigtimes_{p \in P}\lambda_p$ to be the unique Borel measure on $B^D$ satisfying
\[
 \left(\bigtimes_{p \in P}\lambda_p \right)\left(\bigtimes_{p \in P} B_p \right) = \prod_{p \in P} \lambda_p(B_p)
\]
for every choice of $B_p \subseteq B^p$.  (While the notation here may be unconventional, the existence of this product measure is equivalent to a standard result of measure theory.  See, \textit{e.g.}, Section 1.6 in \cite{Cinlar}.)

Besides Cartesian products, we will use the following notation frequently:
\begin{itemize}
 \item $\mathcal{P}(X)$ denotes the power set of a set $X$;
 \item $\mathcal{B}(X)$ denotes the Borel $\sigma$-algebra of a topological space $X$;
 \item $\delta_x$ denotes the standard point mass measure concentrated at a point $x$;
 \item $\mathbf{1}_B$ denotes the indicator function of a set $B$ (so $\mathbf{1}_B(x) = \delta_x(B)$);
 \item $f|_B$ denotes the restriction of a function $f$ to a set $B$;
 \item $[i]$ denotes the von Neumann ordinal $\{0,1,\ldots, i-1\}$ of a positive integer $i$;
 \item $i^-$ and $i^+$ denote, respectively, the immediate predecessor and successor of an integer $i$.
\end{itemize}

\subsection{Model Parameters}

The physical system will be described by the following 5 quantities:

\begin{itemize}
 \item a positive constant $\theta_a$, representing the rate at which objects tend to attach;
 \item a positive constant $\theta_d$, representing the rate at which objects tend to detach;
 \item a positive integer $n$, representing the number of objects;
 \item a positive integer $N$, representing the dimension of the Euclidean space in which the objects move;
 \item a probability measure $\eta$ on $\mathcal{B}(\mathbb{R}^N)$, representing the distribution of the perturbation of newly-attached objects from the centroid.
\end{itemize}

For convenience, we define $E := \mathbb{R}^N$, and $\overline{\eta} := \int_E \mathbf{x}\,d\eta(\mathbf{x})$, which we require to be well-defined and finite.  Note that the wait time for a given detached object to attach is exponentially-distributed with parameter $\theta_a$, and the wait time for a given attached object to detach is exponentially-distributed with parameter $\theta_d$.

\subsection{State Spaces}

The mathematical space describing the location of the $n$ objects at a fixed time is $nN$-dimensional, but the important role played by the centroid of attached objects means that it should probably be explicitly tracked as well (even though in most cases that position can be deduced from the positions of the attached objects).  Additionally, the attached/detached status of each object needs to be tracked.  Thus, the primary state space we use will be
\[
\mathsf{X} := \left\{(\psi,\mathbf{v}) \in \{0,1\}^{[n]} \times E^{[n^+]} : \sum_{i\in [n]} \psi(i)(\mathbf{v}(i)-\mathbf{v}(n))=0\right\}.
\]
The $n$ objects are numbered from $0$ to $n^-$.  For a point $(\psi,\mathbf{v}) \in \mathsf{X}$, $\psi(i)$ represents the attached ($\psi(i)=1$) or detached ($\psi(i)=0$) status of the $i$th object, $\mathbf{v}(i)$ represents the location of the $i$th object, and $\mathbf{v}(n)$ represents the centroid of the attached objects.

We endow $\{0,1\}$ with the discrete topology, $E$ with the Euclidean topology, the product $\{0,1\}^{[n]} \times E^{[n^+]}$ with the corresponding product topology, and the subset $\mathsf{X}$ with the corresponding subset topology. 

In some situations, it will be useful to work with the (much simpler) secondary state space $\hat{\mathsf{X}} := [n^+]$, which represents the number of attached objects.
Given $\psi \in \{0,1\}^{[n]}$, define $|\psi| := \sum_{i \in [n]} \psi(i)$.  The projection $\pi: \mathsf{X} \rightarrow \hat{\mathsf{X}}$ given by the formula $\pi(\psi,\mathbf{v}) := |\psi|$ provides the natural connection between the two state spaces.

\subsection{Kernels}\label{notanddef}

Here we give a formula for a mathematical object $\alpha$ that we will argue corresponds to the CTCM's instantaneous evolution law on $\mathsf{X}$.  In the two subsequent subsections of Section 3, we will prove that $\alpha$ is a rate kernel, and that it generates a Markov process.  The rate kernel $\alpha$ will be the product of a rate function $c$ and a jump transition kernel $\mu$.  On the finite state space $\hat{\mathsf{X}}$, we will define $\hat{\alpha}$, $\hat{c}$, and $\hat{\mu}$, correspondingly. 

\begin{itemize}

 \item Given $\mathbf{a},\mathbf{b} \in E$ and $a, b \in \mathbb{R}$, define the scale-and-translate function $S_{(\mathbf{a},\mathbf{b},a,b)}: E \times E \rightarrow E \times E$ by the formula $S_{(\mathbf{a},\mathbf{b},a,b)}(\mathbf{x},\mathbf{y}) := (a(\mathbf{x}-\mathbf{a}),b(\mathbf{y}-\mathbf{b}))$.
 \item For each $i \in [n]$:
 \begin{itemize}

\item Define $r_i: \{0,1\}^{[n]} \rightarrow (0,1]$ by the formula
\[
 r_i(\psi) := \frac{\theta_d\psi(i)+\theta_a(1-\psi(i))}{\theta_d|\psi|+\theta_a(n-|\psi|)}.
\]
It is a direct consequence of the definition of exponential random variables that the minimum of $n$ independent exponential random variables is itself exponentially-distributed with a parameter that is the sum of the parameters of the $n$ independent variables.  Thus, the wait time for a system whose combined status is represented by $\psi$ to undergo a change of status is exponentially distributed with parameter $\theta_d|\psi|+\theta_a(n-|\psi|)$.  A straightforward calculation shows that $r_i(\psi)$ represents the probability that that next change of status involves the $i$th object.

\item Define $s_i: \{0,1\}^{[n]} \rightarrow \{0,1\}^{[n]}$ so that $s_i(\psi)$ disagrees with $\psi$ precisely on $\{i\}$; \textit{i.e.},  by the formula
$s_i(\psi) := (\psi \setminus \{(i,\psi(i))\}) \cup \{(i,1-\psi(i))\}$.  Whenever the $i$th object changes status, the combined status of the objects goes from $\psi$ to $s_i(\psi)$.

\item Define $P_i$ to be the partition of $[n^+]$ consisting of singletons except for the part $\{i,n\}$; \textit{i.e.}, $P_i :=  \{\{j\}: j \in [n] \setminus \{i\}\} \cup \{\{i,n\}\}$.  The elements of $\{i,n\}$ index the locations that may change when the $i$th object changes status.

\item Define $F_i: E \times E \rightarrow E^{\{i,n\}}$ by the formula $F_i(\mathbf{x},\mathbf{y}) := \{(i,\mathbf{x}),(n,\mathbf{y})\}$, and define $G_i: E \rightarrow E^{\{i\}}$ by the formula $G_i(\mathbf{x}) := \{(i,\mathbf{x})\}$.  $F_i$ and $G_i$ are used to index tuples.

\item Given $(\psi,\mathbf{v}) \in \{0,1\}^{[n]} \times E^{[n^+]}$, define the measure $\mu_{\{i\}}^{(\psi,\mathbf{v})}$ on $E^{\{i\}}$ by the formula $\mu_{\{i\}}^{(\psi,\mathbf{v})} := \delta_{\mathbf{v}(i)} \circ G_i^{-1}$.    
This formula reflects the fact that the $i$th object doesn't move when some other object changes status.

\item Define the measure $\mu_{\{i,n\}}^{(\psi,\mathbf{v})} :=$
\[
  \begin{cases}
 (\delta_{\mathbf{v}(i)} \times \delta_{\mathbf{v}(n)}) \circ F_i^{-1}, & \text{if $|\psi|=\psi(i) = 1$}\\
 (\delta_{\mathbf{v}(i)} \times \delta_{\mathbf{v}(n)}) \circ S_{(\mathbf{0},(\mathbf{v}(i)-\mathbf{v}(n))/|\psi|^-,1,1)}^{-1}\circ F_i^{-1}, & \text{if $|\psi|>\psi(i) = 1$}\\
(\eta \times I) \circ S_{(-\mathbf{v}(n),-|\psi|^+\mathbf{v}(n),1,1/|\psi|^+)}^{-1} \circ F_i^{-1}, & \text{if $\psi(i) = 0$}\\
\end{cases}
\]
on $E^{\{i,n\}}$, where $I: E \times \mathcal{B}(E) \rightarrow [0,1]$ is the inclusion kernel defined by the formula $I(\mathbf{x},C)=\delta_\mathbf{x}(C)$, so $(\eta \times I)(B \times C) = \int_B \eta(d\mathbf{x})I(\mathbf{x},C)=\int_{B \cap C} \eta(d\mathbf{x})=\eta(B \cap C)$.  The formula for $\mu_{\{i,n\}}^{(\psi,\mathbf{v})}$ reflects the various ways that the centroid and the location of the $i$th object can change when that object changes status.  If the given object starts as the only attached one, neither its location nor the ``centroid'' changes when the object detaches.  If the given object starts as one of two or more attached objects, upon its detachment the centroid relocates to correspond to the reduced collection of attached objects, but the location of the given object doesn't change.  If the given object starts as detached, its possible locations upon attachment (and their various likelihoods) are perturbations from the old centroid, as specified by $\eta$; the centroid itself changes to account for the enlarged collection of attached objects.
\end{itemize}

\item Define $\tilde{\mu}: \mathsf{X} \times \mathcal{B}(\{0,1\}^{[n]} \times E^{[n^+]}) \rightarrow [0,\infty)$ by the formula 
\[
\tilde{\mu}((\psi,\mathbf{v}),\cdot) := \sum_{i \in [n]} r_i(\psi) \left( \delta_{s_i(\psi)} \times \bigtimes_{p \in P_i} \mu_p^{(\psi,\mathbf{v})}\right).
\]
The $i$th term in parentheses represents the probability that if the $i$th object is the first to change status, then the new configuration of the system is in the given set.

\item Define $\mu: \mathsf{X} \times \mathcal{B}(\mathsf{X}) \rightarrow [0,\infty)$ to be the restriction of $\tilde{\mu}$ to $\mathsf{X} \times \mathcal{B}(\mathsf{X})$.  Given a starting configuration $\mathsf{x}$, $\mu(\mathsf{x},B)$ represents the probability that the configuration after the next attachment/detachment event will be in $B$.

\item Define $c: \mathsf{X} \rightarrow (0,\infty)$ by the formula
 $c(\psi,\mathbf{v}) :=  \theta_d|\psi| + \theta_a (n-|\psi|)$.  Given a starting configuration $\mathsf{x}$, $c(\mathsf{x})$ represents the reciprocal of the expected wait time until the next attachment/detachment event.  For later convenience, we define $\theta := n\max\{\theta_a,\theta_d\}$, which is an upper bound for $c$.

\item Define $\alpha : \mathsf{X} \times \mathcal{B}(\mathsf{X}) \rightarrow [0,\infty)$ by the formula $\alpha(\mathsf{x},B) := c(\mathsf{x})\mu(\mathsf{x},B)$.

\item Define $\hat{\mu}: \hat{\mathsf{X}} \times \mathcal{P}(\hat{\mathsf{X}})\rightarrow [0,\infty)$ by the formula
\[
 \hat{\mu}(i,\cdot) := \frac{\theta_d i \delta_{i^-} + \theta_a (n-i)\delta_{i^+}}{\theta_di + \theta_a (n-i)}.
\]

\item Define $\hat{c}: \hat{\mathsf{X}} \rightarrow (0,\infty)$ by the formula $\hat{c}(i) := \theta_di + \theta_a (n-i)$.

\item Define $\hat{\alpha} : \hat{\mathsf{X}} \times \mathcal{P}(\hat{\mathsf{X}}) \rightarrow [0,\infty)$ by the formula
 $\hat{\alpha}(i,B) = \hat{c}(i)\hat{\mu}(i,B)$.

\end{itemize}

\subsection{The Jump Transition Kernel}\label{results}

Here we verify that $\mu$ is a probability kernel, which will make $\alpha$ a transition kernel, as desired.  (A probability kernel from $\mathsf{X}$ to $\mathsf{X}$ is a real-valued function $K$ on $\mathsf{X} \times \mathcal{B}(\mathsf{X})$ such that $K(\cdot,B)$ is a measurable function for every $B \in \mathcal{B}(\mathsf{X})$ and such that $K(\mathsf{x},\cdot)$ is a probability measure for every $\mathsf{x} \in \mathsf{X}$.)  In preparation for this verification, we define
\[
 D_i(\psi,\mathbf{v}) := \begin{cases}
 \displaystyle \left(s_i(\psi), \mathbf{v}|_{[n]} \cup \left\{\left(n,\mathbf{v}(n)-\frac{\mathbf{v}(i)-\mathbf{v}(n)}{|\psi|^-}\right)\right\}\right) & \text{if $|\psi| > 1$}\\
 (s_i(\psi),\mathbf{v}) & \text{if $|\psi| \leq 1$,}
 \end{cases}
\]
and
\[
 A_i((\psi,\mathbf{v}),\mathbf{x}) := \left(s_i(\psi),\mathbf{v}|_{[n]\setminus\{i\}} \cup 
 \left\{(i,\mathbf{x}+\mathbf{v}(n)),\left(n,\frac{\mathbf{x}}{|\psi|^+}+\mathbf{v}(n)\right)\right\}\right),
\]
for each $i \in [n]$.  A routine calculation shows that these formulas define functions $D_i: \mathsf{X} \rightarrow \mathsf{X}$ and $A_i: \mathsf{X} \times E \rightarrow \mathsf{X}$.  The informal description of the CTCM suggests that $D_i(\mathsf{x})$ should represent the state of the system immediately after a system in state $\mathsf{x}$ undergoes detachment of the $i$th object, and $A_i(\mathsf{x},\mathbf{x})$ should represent the new state of the system immediately after a system in state $\mathsf{x}$ undergoes attachment of the $i$th object with perturbation $\mathbf{x}$.  These functions come up naturally in the analysis of the formal version of the CTCM here and again in Subsection \ref{integrals}.

\begin{proposition}\label{prop1}The map $\mu$ is a probability kernel from $\mathsf{X}$ to $\mathsf{X}$.\end{proposition}

\begin{proof}
 Since $s_i$ is measurable, $((\psi,\mathbf{v}),B) \mapsto \delta_{s_i(\psi)}(B) = (\mathbf{1}_B \circ s_i)(\psi)$ is a kernel from $\mathsf{X}$ to $\{0,1\}^{[n]}$ for each $i \in [n]$.  Lemma 1.41 of \cite{Kallenberg} implies that finite products of kernels, with or without integration of the parameters, are kernels.  A similar analysis, combined with the fact that translations and dilations of Borel subsets of Euclidean space are Borel subsets of Euclidean space, implies that $(\mathsf{x},B) \mapsto \mu_p^\mathsf{x}(B)$ is a kernel from $\mathsf{X}$ to $E^p$ for every $i \in [n]$ and every $p \in P_i$.  Two more applications of that lemma imply that
 \[
  ((\psi,\mathbf{v}),B) \mapsto \left(\delta_{s_i(\psi)} \times \bigtimes_{p \in P_i} \mu_p^{(\psi,\mathbf{v})}\right)(B)
 \]
 is a kernel from $\mathsf{X}$ to $\{0,1\}^{[n]} \times E^{[n^+]}$.  Since $r_i$ is nonnegative and measurable, and sums of kernels are kernels (see \cite{Kallenberg}), and the restriction of a measure to the measurable subsets of a fixed measurable set is a measure, we can conclude that $\mu$ is a kernel from $\mathsf{X}$ to $\mathsf{X}$.
 
 It remains to show that $\mu((\psi,\mathbf{v}),\mathsf{X})=1$ for every $(\psi,\mathbf{v}) \in \mathsf{X}$.  Fix such $(\psi,\mathbf{v})$, fix $i \in [n]$, and let $\lambda := \delta_{s_i(\psi)} \times \bigtimes_{p \in P_i} \mu_p^{(\psi,\mathbf{v})}$.  
If $\psi(i)=1$, then $\lambda = \delta_{D_i(\psi,\mathsf{v})}$, so $\lambda$ is a probability measure on $\mathsf{X}$. 
If $\psi(i)=0$, note that $\eta \times I$ is a probability measure concentrated on the diagonal in $E \times E$, which implies that  $\mu^{(\psi,\mathbf{v})}_{\{i,n\}}$ is a probability measure concentrated on the set
 \[
  B :=\{ \{(i,\mathbf{x}),(n,(|\psi|\mathbf{v}(n)+\mathbf{x})/|\psi|^+)\}:\mathbf{x}\in E\}.
 \]
 The measure $\lambda$ is the product of $\mu^{(\psi,\mathbf{v})}_{\{i,n\}}$ and point mass measures and is therefore a probability measure on $\{0,1\}^{[n]} \times E^{[n^+]}$.  The information on where $\lambda$'s component measures are concentrated tells us that $\lambda$ itself is concentrated on the set
\[
 \{(s_i(\psi),\mathbf{v}|_{[n]\setminus \{i\}}\cup \mathbf{w}) \in \{0,1\}^{[n]} \times E^{[n^+]}: \mathbf{w}\in B\} = \{A_i((\psi,\mathbf{v}),\mathbf{x}): \mathbf{x} \in E\},
\]
which is a subset of $\mathsf{X}$, the codomain of $A_i$, so $\lambda$ is actually a probability measure on $\mathsf{X}$.

In both cases, $\lambda$ is a probability measure on $\mathsf{X}$.  Letting $i$ (and therefore $\lambda$) vary, we see (since $\sum_{i \in [n]} r_i(\psi) = 1$) that $\mu((\psi,\mathbf{v}),\cdot)$ is a convex combination of probability measures on $\mathsf{X}$ and is therefore itself a probability measure on $\mathsf{X}$. 
\end{proof}

\subsection{Existence of the Markov Process}\label{process}

Using results from the theory of Markov processes, we can now show that $\alpha$ generates one and can give a partial description of its structure.  This essentially demonstrates that the CTCM is well-posed.

\begin{proposition}\label{prop2} For any Borel probability measure $\rho$ on $\mathsf{X}$, there is a discrete-time Markov process $Y$ on $\mathsf{X}$ with transition kernel $\mu$ such that $Y_0$ is $\rho$-distributed.\end{proposition}

\begin{proof}
 Since \cref{prop1} shows that $\mu$ is a probability kernel from $\mathsf{X}$ to $\mathsf{X}$, this follows from Theorem 3.4.1 in \cite{MeynTweedie}. 
\end{proof}

\begin{proposition}\label{prop3}  For every Borel probability measure $\rho$ on $\mathsf{X}$, there is a pure jump-type continuous-time Markov process $X$ on $\mathsf{X}$ with rate kernel $\alpha$ such that $X_0$ is $\rho$-distributed.  If $Y$ is as defined in \cref{prop2}, $(\gamma_i)$ is a sequence of standard exponential random variables, and $\{Y,\gamma_1,\gamma_2,\gamma_3,\ldots\}$ is an independent family, then $X$ can be defined by the formula $X_t = Y_k$ for $t \in [\tau_k,\tau_{k^+})$, where $\tau_k := \sum_{i=1}^k (\gamma_i/c(Y_{i^-}))$.\end{proposition}

\begin{proof}
 Since (by \cref{prop1}) $\mu$ is a kernel from $\mathsf{X}$ to $\mathsf{X}$, and since the formula for $c$ indicates that it is positive and measurable, $\alpha$ is also a kernel from $\mathsf{X}$ to $\mathsf{X}$.
 
For each $i \in [n]$, we have $s_i(\psi)(i) \neq \psi(i)$, so $\delta_{s_i(\psi)}(\{\psi\}) = 0$, which means that
\[
\alpha((\psi,\mathbf{v}),\{(\psi,\mathbf{v})\}) = c(\psi,\mathbf{v})\mu((\psi,\mathbf{v}),\{(\psi,\mathbf{v})\}) = 0
\]
for every $(\psi,\mathbf{v}) \in \mathsf{X}$.  

Given $\rho$, let $Y$ be as in \cref{prop2}.  Also, let $(\gamma_k)$ be a sequence of standard exponential random variables such that $\{Y,\gamma_1,\gamma_2,\gamma_3,\ldots\}$ is an independent family.  (That, without loss of generality, such a sequence can be assumed to exist is a consequence of the Ionescu Tulcea Theorem.  See, \textit{e.g.}, Corollary 6.18 in \cite{Kallenberg}.) 

Suppose that $\sum_k (\gamma_k/c(Y_{k^-}))$ converges.  The formula for $c$ guarantees that it has a bounded range, so $\sum_k \gamma_k$ converges; thus, its partial sums must be bounded.  This implies that the average term in a partial sum goes to $0$ as the index of the partial sum goes to $\infty$.  By The Strong Law of Large Numbers, this fails almost surely.  Hence, $\sum_k (\gamma_k/c(Y_{k^-})) = \infty$ almost surely.

Because of the preceding observations, Theorem 12.18 in \cite{Kallenberg} yields the desired result. 
 \end{proof}
 
 \section{Long-time Behavior of the CTCM}\label{sec:ltbehavior}
 
 
With the CTCM established as well-posed, we proceed to analyze aspects of its long-time behavior.  In Subsection \ref{subsecconnect} we establish a version of Dynkin's Criterion, and in Subsection \ref{fsprocess} we use it to show that projection $\pi$ induces a finite-state process that is naturally connected to the CTCM.  In Subsection \ref{ltbfsprocess}, we show that the finite-state process has an attracting steady-state.  In Subsection \ref{integrals}, we derive a formula for integrating with respect to $\mu$.  In Subsection \ref{secgrowth}, we derive some elementary evolutionary bounds for the CTCM.  In Subsection \ref{integrability}, we show that if the CTCM has an integrable initial distribution, its distribution remains integrable for all time.  In Subsection \ref{continuity}, we show that (while its sample paths are discontinuous), the CTCM's centroid has an expected value that varies continuously with time.  Finally, in Subsection \ref{velocity} we rigorously compute the velocity of the expected centroid location.



\subsection{Connecting the finite-state system and the CTCM}\label{subsecconnect}

The finite-state system corresponding to the auxiliary kernel $\hat{\alpha}$ plays a vital role in the analysis of the CTCM.  To connect these two systems carefully, we need a version of 
Dynkin's Criterion geared towards discrete-time Markov processes on topological spaces.  This is the content of the following lemma.

\begin{lemma}\label{Dynkin} Let $T_1$ and $T_2$ be topological spaces, let $f: T_1 \rightarrow T_2$ be a continuous surjection having a continuous right-inverse, let $Q_1: T_1 \times \mathcal{B}(T_1) \rightarrow [0,1]$ be a probability kernel, and let $Q_2: T_2 \times \mathcal{B}(T_2) \rightarrow [0,1]$ be a function satisfying $Q_1(x,f^{-1}(B))=Q_2(f(x),B)$ for every $x \in T_1$ and $B \in \mathcal{B}(T_2)$.

Then $Q_2$ is a probability kernel, and for every probability measure $\rho$ on $\mathcal{B}(T_1)$ and every discrete-time Markov process $Z$ with initial distribution $\rho$ and transition kernel $Q_1$, the discrete-time stochastic process $f \circ Z$ is a Markov process with initial distribution $\rho \circ f^{-1}$ and transition kernel $Q_2$.
\end{lemma}

\begin{proof}
 Given $y \in T_2$, $y=f(x)$ for some $x \in T_1$ because $f$ is surjective.  Then, by hypothesis, $Q_2(y,\cdot) = Q_2(f(x),\cdot)=Q_1(x,\cdot) \circ f^{-1}$, which is a measure on $\mathcal{B}(T_2)$, since $f$ is continuous and therefore measurable.  Also, $Q_2(y,T_2)=Q_1(x,f^{-1}(T_2))=Q_1(x,T_1)=1$, so $Q_2(y,\cdot)$ is a probability measure.

 Let $g$ be a continuous right-inverse of $f$.  For any 
 $B \in \mathcal{B}(T_2)$, we have $Q_2(\cdot,B)=Q_2((f\circ g)(\cdot),B)=Q_2(f(g(\cdot)),B)=Q_1(g(\cdot),f^{-1}(B))=Q_1(\cdot,f^{-1}(B))\circ g$.  Since $g$ is continuous, it is measurable, so $Q_2(\cdot,B)=Q_1(\cdot,f^{-1}(B))\circ g$ is also measurable.  Thus, $Q_2$ is a probability kernel.
 
Let $\rho$ be a probability measure on $\mathcal{B}(T_1)$, and let $Z$ be a discrete-time Markov process with initial distribution $\rho$ and transition kernel $Q_1$.  Let $\lambda = \rho \circ f^{-1}$.  For clarity below, we write $Q_{1,x}$ for $Q_1(x,\cdot)$, and $Q_{2,y}$ for $Q_2(y,\cdot)$.  Note that the hypothesized relationship between $Q_1$ and $Q_2$ tells us that $Q_{2,f(x)}=Q_{1,x} \circ f^{-1}$ for every $x \in T_1$.

Let a nonnegative integer $j$ and sets $B_0, \ldots, B_j \in \mathcal{B}(T_2)$ be given.  Using Theorem 3.4.1 in \cite{MeynTweedie} and the change-of-variable formula for measure-theoretic integration (see, \textit{e.g.}, Theorem 5.2 in Chapter 1 of \cite{Cinlar}) with $y_i=f(x_i)$, we have
\begin{multline*}
 \mathbb{P}((f \circ Z)_0 \in B_0, \ldots, (f \circ Z)_j \in B_j) 
 = \mathbb{P}(f(Z_0) \in B_0, \ldots, f(Z_j) \in B_j)\\
 = \mathbb{P}(Z_0 \in f^{-1}(B_0), \ldots, Z_j \in f^{-1}(B_j)) \\
 = \int_{x_0 \in f^{-1}(B_0)} \cdots \int_{x_{j^-} \in f^{-1}(B_{j^-})}
 \rho(dx_0)Q_{1,x_0}(dx_1)\cdots Q_{1,x_{j^-}}(f^{-1}(B_j)) \\
 = \int \cdots \int \mathbf{1}_{f^{-1}(B_0)}(x_0)\cdots \mathbf{1}_{f^{-1}(B_{j^-})}(x_{j^-}) \rho(dx_0)Q_{1,x_0}(dx_1)\cdots Q_{1,x_{j^-}}(f^{-1}(B_j))\\
 = \int \cdots \int \mathbf{1}_{B_0}(f(x_0))\cdots \mathbf{1}_{B_{j^-}}(f(x_{j^-})) \rho(dx_0)Q_{1,x_0}(dx_1)\cdots Q_{1,x_{j^-}}(f^{-1}(B_j))\\
 = \int \cdots \int \mathbf{1}_{B_0}(y_0)\cdots \mathbf{1}_{B_{j^-}}(y_{j^-}) \lambda(dy_0)Q_{2,y_0}(dy_1)\cdots Q_{2,y_{j^-}}(B_j)\\
 = \int_{y_0 \in B_0} \cdots \int_{y_{j^-} \in B_{j^-}}\lambda(dy_0)Q_{2,y_0}(dy_1)\cdots Q_{2,y_{j^-}}(B_j).
\end{multline*}
This equation, and another application of Theorem 3.4.1 in \cite{MeynTweedie}, tells us that $f \circ Z$ is a Markov process with initial distribution $\lambda$ and transition kernel $Q_2$, as desired. 
\end{proof}

\subsection{Existence of the Finite-State Process}\label{fsprocess}

\begin{proposition}\label{prop4} If $X$ is as in \cref{prop3}, then $\hat{X} := \pi \circ X$ is a pure jump-type continuous-time Markov process with rate kernel $\hat{\alpha}$ and initial distribution $\rho \circ \pi^{-1}$.\end{proposition}

\begin{proof}
By definition of $\hat{X}$ in terms of $X$, the former has initial distribution $\rho \circ \pi^{-1}$ because the latter has initial distribution $\rho$.
 Because of $X$'s definition in terms of $Y$ from \cref{prop2}, $\hat{X}_t=\hat{Y}_k$ for $t \in [\tau_k,\tau_{k^+})$, where $\hat{Y} := \pi \circ Y$, $\tau_k := \sum_{i=1}^k(\gamma_i/\hat{c}(\hat{Y}_{i^-}))$, and $(\gamma_i)$ is a sequence of independent standard exponential random variables that are independent from $Y$.  Because $\hat{\alpha}(i,\{i\})= \hat{c}(i)\hat{\mu}(i,\{i\})=0$ for every $i \in \hat{\mathsf{X}}$ (by the formula for $\hat{\mu}$), Theorem 12.18 in \cite{Kallenberg} will show that $\hat{X}$ is a pure jump-type continuous-time Markov process with rate kernel $\hat{\alpha}$ if we can show that $\hat{Y}$ is a discrete-time Markov process with transition kernel $\hat{\mu}$.  
 
 To this end, we will use the version of Dynkin's Criterion presented in Lemma \ref{Dynkin}.  Note that $\pi$ is a continuous surjection.  Let $\mathbf{\zeta}$ be the zero element of $E^{[n^+]}$, define $g: \hat{\mathsf{X}} \rightarrow \mathsf{X}$ by the formula $g(i) := (\mathbf{1}_{[i]},\mathbf{\zeta})$, and note that $g$ is a continuous right-inverse of $\pi$.  The last (and main) hypothesis of Dynkin's Criterion is that $\mu((\psi,\mathbf{v}),\pi^{-1}(B))=\hat{\mu}(|\psi|,B)$ for every $(\psi,\mathbf{v}) \in \mathsf{X}$ and every $B \subseteq \hat{\mathsf{X}}$.  Since $\hat{\mu}(|\psi|,\cdot)$ is additive for every $\psi$, it suffices to consider $B$ of the form $\{j\}$.  Note that
\begin{multline*}
 \mu((\psi,\mathbf{v}),\pi^{-1}(\{j\}))
 = \sum_{i \in [n]} r_i(\psi)\delta_{|s_i(\psi)|}(\{j\})\\
 = \sum_{i \in \psi^{-1}(\{1\})} r_i(\psi)\delta_{|s_i(\psi)|}(\{j\})
 + \sum_{i \in \psi^{-1}(\{0\})} r_i(\psi)\delta_{|s_i(\psi)|}(\{j\})
 \\
 = \frac{1}{\theta_d|\psi|+\theta_a(n-|\psi|)}\left(\sum_{i \in \psi^{-1}(\{1\})} \theta_d\delta_{|\psi|}(\{j^+\})
 + \sum_{i \in \psi^{-1}(\{0\})} \theta_a\delta_{|\psi|}(\{j^-\})\right)
 \\
  = \frac{\theta_d|\psi|\delta_{|\psi|}(\{j^+\})+\theta_a(n-|\psi|)\delta_{|\psi|}(\{j^-\})}{\theta_d|\psi|+\theta_a(n-|\psi|)}
  = \frac{\theta_d|\psi|\delta_{|\psi|^-}+\theta_a(n-|\psi|)\delta_{|\psi|^+}}{\theta_d|\psi|+\theta_a(n-|\psi|)}(\{j\})\\ = \hat{\mu}(|\psi|,\{j\}).
\end{multline*}
Dynkin's Criterion therefore implies that $\hat{Y}$ is a discrete-time Markov process with transition kernel $\hat{\mu}$. 
\end{proof}

\subsection{Long-Time Behavior of the Finite-State Process}\label{ltbfsprocess}

\begin{proposition}\label{prop5}  The unique invariant distribution $\sigma$ for the rate kernel $\hat{\alpha}$ is given by the formula
\begin{equation}
\label{sigmadef}
 \sigma := \frac{1}{(\theta_d+\theta_a)^n}\sum_{k \in \hat{\mathsf{X}}} \binom{n}{k}\theta_d^{n-k}\theta_a^k\delta_k. 
\end{equation}
If $\hat{Z}$ is a pure jump-type continuous-time Markov process with rate kernel $\hat{\alpha}$, then the distribution of $\hat{Z}_t$ converges to $\sigma$ as $t \rightarrow \infty$, regardless of the distribution of $\hat{Z}_0$.  \end{proposition}
\begin{proof}
 It is straightforward to check that 
 \[
  \sum_{k \in \hat{\mathsf{X}}} \left(\binom{n^-}{k^-}\theta_a^{k^-}\theta^{n-k^-}_d+\binom{n^-}{k}\theta_a^k\theta^{n-k}_d\right)\delta_k
 \]
 is an invariant measure for the transition kernel $\hat{\mu}$ (where we take $\tbinom{a}{b} := 0$ if $b<0$ or $b>a$).  Using Proposition 12.23 in \cite{Kallenberg} and simplifying, we can deduce that $\sigma$ as defined above is an invariant distribution corresponding to rate kernel $\hat{\alpha}$.  Since $\hat{\mu}$ is irreducible, so is $\hat{\alpha}$, so by Proposition 12.25 in \cite{Kallenberg}, $\sigma$ is the unique such distribution, and it is attracting, as in the theorem statement. 
\end{proof}

From here on, we take $\sigma$ to be defined by \eqref{sigmadef}.

\subsection{Integration Formula}\label{integrals}

The following lemma facilitates future calculations, and (in light of the heuristic interpretation of $D_i$ and $A_i$) provides a validation of the formal version of the CTCM.

\begin{lemma} \label{integrallem}
Suppose $f: \mathsf{X} \rightarrow [0,\infty]$ is measurable and $\mathsf{x} \in \mathsf{X}$.  Then
\[
 \int_\mathsf{X} f(\mathsf{y})\mu(\mathsf{x},d\mathsf{y})
 = \sum_{i \in \psi^{-1}(\{1\})} r_i(\psi)f(D_i(\mathsf{x}))
 + \sum_{i \in \psi^{-1}(\{0\})} r_i(\psi) \int_E f(A_i(\mathsf{x},\mathbf{x}))\,d\eta(\mathbf{x}).
\]
\end{lemma}

\begin{proof}
 Let $x = (\psi,\mathbf{v})$.  By definition of $\mu$,
\begin{align}
 \int_\mathsf{X} f(\mathsf{y})\mu(\mathsf{x},d\mathsf{y})
 &=\sum_{i \in [n]} r_i(\psi) \int_\mathsf{X} f(\mathsf{y})\,d \left(\delta_{s_i(\psi)} \times \bigtimes_{p \in P_i} \mu_p^{(\psi,\mathbf{v})}\right)(\mathsf{y}) \notag\\
 &=\sum_{i \in [n]} r_i(\psi) \int_{E^{[n^+]}} f(s_i(\psi),\mathbf{w})\,d \left(\bigtimes_{p \in P_i} \mu_p^{(\psi,\mathbf{v})}\right)(\mathbf{w}) \notag\\
 &=\sum_{i \in [n]} r_i(\psi) \int_{E^{\{i,n\}}} f(s_i(\psi),\mathbf{v}|_{[n]\setminus\{i\}} \cup \mathbf{z})
 \,d\mu_{\{i,n\}}^{(\psi,\mathbf{v})}(\mathbf{z}). \label{6.1}
\end{align}

Suppose $|\psi|=\psi(i)=1$.  Then $\mu_{\{i,n\}}^{(\psi,\mathbf{v})} = \delta_{\mathbf{v}|_{\{i,n\}}}$, so
\begin{multline}
 \int_{E^{\{i,n\}}} f(s_i(\psi),\mathbf{v}|_{[n]\setminus\{i\}} \cup \mathbf{z})\,d\mu_{\{i,n\}}^{(\psi,\mathbf{v})}(\mathbf{z}) \\
 = \int_{E^{\{i,n\}}} f(s_i(\psi),\mathbf{v}|_{[n]\setminus\{i\}} \cup \mathbf{z})\,d\delta_{\mathbf{v}|_{\{i,n\}}}(\mathbf{z}) 
 = f(s_i(\psi),\mathbf{v})
 = f(D_i(\mathsf{x})). \label{6.2}
\end{multline}

Suppose, instead, that $|\psi|>\psi(i)=1$.  Then
\[
\mu_{\{i,n\}}^{(\psi,\mathbf{v})} = \delta_{\{(i,\mathbf{v}(i)),(n,\mathbf{v}(n)-(\mathbf{v}(i)-\mathbf{v}(n))/|\psi|^-)\}},
\]
so
\begin{multline}
  \int_{E^{\{i,n\}}} f(s_i(\psi),\mathbf{v}|_{[n]\setminus\{i\}} \cup \mathbf{z})\,d\mu_{\{i,n\}}^{(\psi,\mathbf{v})}(\mathbf{z}) \\
  = \int_{E^{\{i,n\}}} f(s_i(\psi),\mathbf{v}|_{[n]\setminus\{i\}} \cup \mathbf{z})\,d\delta_{\{(i,\mathbf{v}(i)),(n,\mathbf{v}(n)-(\mathbf{v}(i)-\mathbf{v}(n))/|\psi|^-)\}}(\mathbf{z}) \\
 = f\left(s_i(\psi),\mathbf{v}|_{[n]} \cup \left\{\left(n,\mathbf{v}(n)-\frac{\mathbf{v}(i)-\mathbf{v}(n)}{|\psi|^-}\right)\right\}\right)
 = f(D_i(\mathsf{x})). \label{6.3}
\end{multline}

Finally, suppose that $\psi(i)=0$.  Given $(\mathbf{x},\mathbf{y}) \in E \times E$, we have
\[
 (F_i \circ S_{(-\mathbf{v}(n),-|\psi|^+\mathbf{v}(n),1,1/|\psi|^+)})(\mathbf{x},\mathbf{y}) = \left(\mathbf{x}+\mathbf{v}(n),\frac{\mathbf{y}}{|\psi|^+}+\mathbf{v}(n)\right),
\]
so
\begin{multline}
  \int_{E^{\{i,n\}}} f(s_i(\psi),\mathbf{v}|_{[n]\setminus\{i\}} \cup \mathbf{z})\,d\mu_{\{i,n\}}^{(\psi,\mathbf{v})}(\mathbf{z}) \\
  = \int_{E \times E}\!
  f\left(s_i(\psi),\mathbf{v}|_{[n]\setminus\{i\}} \cup \left\{(i,\mathbf{x}+\mathbf{v}(n)),\left(n,\frac{\mathbf{y}}{|\psi|^+}+\mathbf{v}(n)\right)\right\}\right)\,d(\eta \times I)(\mathbf{x},\mathbf{y}) \\
  = \int_E\!
  d\eta(\mathbf{x}) \int_E\!\! I(\mathbf{x},d\mathbf{y})f\left(s_i(\psi),\mathbf{v}|_{[n]\setminus\{i\}} \cup \left\{(i,\mathbf{x}+\mathbf{v}(n)),\left(n,\frac{\mathbf{y}}{|\psi|^+}+\mathbf{v}(n)\right)\right\}\right) \\
  = \int_E f\left(s_i(\psi),\mathbf{v}|_{[n]\setminus\{i\}} \cup \left\{(i,\mathbf{x}+\mathbf{v}(n)),\left(n,\frac{\mathbf{x}}{|\psi|^+}+\mathbf{v}(n)\right)\right\}\right)\,d\eta(\mathbf{x})\\
  = \int_E f(A_i(\mathsf{x},\mathbf{x}))\,d\eta(\mathsf{x}).
  \label{6.4}
\end{multline}

Substituting \eqref{6.2}, \eqref{6.3}, and \eqref{6.4} into \eqref{6.1} gives the desired formula. 
\end{proof}

\subsection{Bounds on Growth and Movement}\label{secgrowth}

In the remaining subsections of Section 4, we focus on location in the physical space $E$.  For this reason, we introduce some additional notation:
\begin{itemize}
\item $|\cdot|$ is the $\infty$-norm on $E$;
\item for each $i \in [n^+]$, $f_i: \mathsf{X} \rightarrow E$ is the projection $f_i(\psi,\mathbf{v}) := \mathbf{v}(i)$;
\item $g: \mathsf{X} \rightarrow [0,\infty]$ is defined by the formula $g(\mathsf{x}) := \max
\{|f_i(\mathsf{x})|: i \in [n^+]\}$.
\end{itemize}
From here forward, we also assume that $\eta$ is supported on a compact set, and pick $R>0$ such that $\eta$ is supported on $\{\mathbf{x} \in E: |\mathbf{x}|\leq R\}$.

\begin{lemma}\label{lemgrowth}
 Let $Y$ be as in Proposition \ref{prop2}, and let $k$ be a whole number.  Then $g(Y_{k^+}) \leq g(Y_k)+R$ almost surely.  Therefore, by induction, $g(Y_k) \leq g(Y_0) + kR$ almost surely.
\end{lemma}

\begin{proof}
 The formulas for $D_i$ and $A_i$ indicate that $g(D_i(\mathsf{x})) \leq g(\mathsf{x})$ and $g(A_i(\mathsf{x},\mathbf{x})) \leq g(\mathsf{x})+|\mathbf{x}|$ for every $\mathsf{x} \in \mathsf{X}$ and $\mathbf{x} \in E$ and $i \in [n]$.  Therefore, applying Lemma \ref{integrallem} with $f := \mathbf{1}_{\{\mathsf{y}:g(\mathsf{y})\leq g(\mathsf{x}) + R\}}$ gives
\begin{multline*}
\mu(\mathsf{x},\{\mathsf{y}:g(\mathsf{y})\leq g(\mathsf{x}) + R\})\\
= \!\!\!
\sum_{\substack{i \in \psi^{-1}(\{1\})\\
g(D_i(\mathsf{x})) \leq g(\mathsf{x})+R}} 
\!\!\!r_i(\psi) + \!\!\!
\sum_{i \in \psi^{-1}(\{0\})} \!\!\! 
r_i(\psi) \eta(\{\mathbf{x} \in E: g(A_i(\mathsf{x},\mathbf{x}))\leq g(\mathsf{x})+R\})\\
\geq \sum_{i \in \psi^{-1}(\{1\})} \!\!r_i(\psi) +\!\! \sum_{i \in \psi^{-1}(\{0\})} \!\! r_i(\psi) \eta(\{\mathbf{x} \in E: |\mathbf{x}|\leq R\})
\\
= \sum_{i \in \psi^{-1}(\{1\})} \!\! r_i(\psi) +\!\! \sum_{i \in \psi^{-1}(\{0\})}\!\! r_i(\psi) 
= \sum_{i \in [n]} r_i(\psi) = 1, 
\end{multline*}
so
\begin{equation}
 \mu(\mathsf{x},\{\mathsf{y}:g(\mathsf{y}) \leq g(\mathsf{x}) + R\})=1.
 \label{int}
\end{equation}

Let $\lambda$ be the distribution of $Y_k$.  By Proposition 8.2 in \cite{Kallenberg}, the distribution of $(Y_k,Y_{k^+})$ is $\lambda \times \mu$, so \eqref{int} implies that   
\begin{multline*}
 \mathbb{P}\{g(Y_{k^+})\leq g(Y_k) + R\} 
 = (\lambda \times \mu)(\{(\mathsf{x},\mathsf{y}) \in \mathsf{X} \times \mathsf{X}: g(\mathsf{y}) \leq g(\mathsf{x}) + R\})\\
 = \int_{\mathsf{x} \in \mathsf{X}}\! \lambda(d\mathsf{x}) \mu(\mathsf{x},\{\mathsf{y}:g(\mathsf{y}) \leq g(\mathsf{x}) + R\})
 = \int_{\mathsf{x} \in \mathsf{X}}\! \lambda(d\mathsf{x})= 1.
\end{multline*}
\end{proof}

\begin{lemma}\label{lemiterate}
 Let $Y$ be as in Proposition \ref{prop2}, and let $k_1$ and $k_2$ be whole numbers.  Then $|f_n(Y_{k_2})-f_n(Y_{k_1})| \leq 2g(Y_0)+(k_1+k_2)R$ almost surely.
\end{lemma}

\begin{proof}
From the triangle inequality and Lemma \ref{lemgrowth},
\begin{multline*}
 |f_n(Y_{k_2})-f_n(Y_{k_1})|
 \leq |f_n(Y_{k_2})|+|f_n(Y_{k_1})|\\
 \leq g(Y_{k_2}) + g(Y_{k_1})
 \leq g(Y_0)+k_2R+g(Y_0)+k_1R = 2g(Y_0)+(k_1+k_2)R.
\end{multline*}
\end{proof}

\subsection{Global Integrability}\label{integrability}

Here we prove that integrability of the initial location of the objects and their centroid entails their integrability at all later times.  We split off a small part of the argument for reuse.

\begin{lemma}\label{lemexpo}
 The probability that the sum of $k$ independent standard exponential random variables is less than or equal to $C \geq 0$ is no greater than $C^k/(k!)$.
\end{lemma}

\begin{proof}
 Such a sum has a probability density function of $x \mapsto e^{-x}x^{k^-}/((k^-)!)$ for $x \geq 0$.
(See, \textit{e.g.}, Proposition 3.1 in Chapter 6 of \cite{Ross}.) Thus, the specified probability is
\[
\int_0^C \frac{e^{-x}x^{k^-}}{(k^-)!}\,dx
\leq \int_0^C \frac{x^{k^-}}{k^-!}\,dx
= \frac{C^k}{k!}.
\]
\end{proof}

\begin{proposition}\label{globalintegrability} 
  Let $\rho$ be a distribution on $\mathsf{X}$ such that $f_i$ is $\rho$-integrable for every $i$.  Let $X$ be as in \cref{prop3}. Then for every $i \in [n^+]$ and $t \geq 0$, $\mathbb{E}(f_i(X_t))$ is well-defined and finite.
\end{proposition}

\begin{proof}
  Let $Y$, $(\gamma_k)$, and $(\tau_k)$ be as in Proposition \ref{prop2} and Proposition \ref{prop3}.  Fix $t \geq 0$, and note that Lemma \ref{lemgrowth} implies that
\begin{align}
 \mathbb{E}(g(X_t)) &= \sum_{k=0}^\infty \mathbb{E}(g(X_t) \mid t \in [\tau_k,\tau_{k^+}))\mathbb{P}(t \in [\tau_k,\tau_{k^+})) \notag \\ 
 &= \sum_{k=0}^\infty \mathbb{E}(g(Y_k) \mid t \in [\tau_k,\tau_{k^+}))\mathbb{P}(t \in [\tau_k,\tau_{k^+})) \notag\\
 &\leq \sum_{k=0}^\infty \mathbb{E}(g(Y_0)+kR \mid t \in [\tau_k,\tau_{k^+}))\mathbb{P}(t \in [\tau_k,\tau_{k^+})) \notag\\
 &= \sum_{k=0}^\infty \mathbb{E}(g(Y_0) \mid t \in [\tau_k,\tau_{k^+}))\mathbb{P}(t \in [\tau_k,\tau_{k^+})) +R\sum_{k=0}^\infty k\mathbb{P}(t \in [\tau_k,\tau_{k^+})) \notag\\
 &= \mathbb{E}(g(Y_0))+R\sum_{k=0}^\infty k\mathbb{P}(t \in [\tau_k,\tau_{k^+})) \notag\\
 &\leq \mathbb{E}(g(X_0)) + R\sum_{k=0}^\infty k \mathbb{P}(t \geq \tau_k).
\label{intexpand}
\end{align}

Since each $f_i$ is $\rho$-integrable, so is $g$, so
$
\mathbb{E}(g(X_0)) = \int (g \circ X_0)d\mathbb{P} = \int g d\rho < \infty$.
By the formulas for $\tau_k$ and $c$, we have
\begin{equation}
\tau_k = \sum_{i=1}^k \frac{\gamma_i}{c(Y_{i^-})} \geq \frac{1}{\theta}\sum_{i=1}^k \gamma_i \label{taubound}
\end{equation}
for $k \geq 1$.
By Lemma \ref{lemexpo}, \eqref{intexpand} and \eqref{taubound} imply that
\[
 \mathbb{E}(g(X_t)) 
 \leq \mathbb{E}(g(X_0)) + R\sum_{k=1}^\infty k \frac{(\theta t)^k}{k!}
= \mathbb{E}(g(X_0)) + R\theta te^{\theta t} < \infty.
\]
By the definition of $g$, this estimate shows that $\mathbb{E}(f_i(X_t))$ is defined and finite for every $i \in [n^+]$.  
\end{proof}

\subsection{Continuity}\label{continuity}

\begin{lemma}\label{explem}
 If $\gamma$ is a standard exponential random variable, $\xi$ is a real-valued random variable, and $B$ is a non-negligible event such that $(\xi,B)$ is independent of $\gamma$, and $a < b$ are real numbers, then
\[
\mathbb{P}(a+\xi < \gamma \leq b+\xi \bigm| B) \leq b-a.
\]
\end{lemma}

\begin{proof}
 Since $\gamma$'s probability density function is bounded by $1$, the probability that $\gamma$ is in a specified \textit{deterministic} interval is bounded by the width of the interval.  Given a positive integer $M$, use this fact, the law of total probability, and the hypothesized independence to deduce that
\begin{multline*}
 \mathbb{P}(\{a+ \xi < \gamma \leq b+\xi\} \cap B)\\
= \sum_{j=-\infty}^{\infty}
 \mathbb{P}\left(\{-b \leq \xi-\gamma < -a\} \cap B \biggm|
 j \leq \frac{M\gamma}{b-a} < j^+\right)
 \mathbb{P}\left\{j \leq \frac{M\gamma}{b-a} < j^+\right\}\\
\leq \frac{b-a}{M}\sum_{j=-\infty}^{\infty}
 \mathbb{P}\left(\left\{\frac{(b-a)j}{M}-b \leq \xi < \frac{(b-a)j^+}{M}-a\right\} \cap B \biggm|
 j \leq \frac{M\gamma}{b-a} < j^+\right)\\
= \frac{b-a}{M}\sum_{j=-\infty}^{\infty}
 \mathbb{P}\left(\left\{\frac{(b-a)j}{M}-b \leq \xi < \frac{(b-a)j^+}{M}-a\right\} \cap B \right)\\
= \frac{b-a}{M}\sum_{j=-\infty}^{\infty}
 \mathbb{P}\left(\left\{j-M \leq \frac{M(\xi+a)}{b-a} < j+1\right\} \cap B \right)\\
= \frac{(b-a)(M+1)}{M}\mathbb{P}(B).
\end{multline*}
Letting $M \rightarrow \infty$ gives the desired formula.
\end{proof}

\begin{proposition}\label{continuityprop} 
  Let $\rho$ be a distribution on $\mathsf{X}$ such that $f_i$ is $\rho$-integrable for every $i$.  Let $X$ be as in \cref{prop3}. Then $t \mapsto \mathbb{E}(f_n(X_t))$ is continuous.
\end{proposition}

\begin{proof}
  Let $t_2 \geq t_1 \geq 0$ be given.  Let $Y$, $(\gamma_k)$, and $(\tau_k)$ be as in Proposition \ref{prop3}.  Given whole numbers $i,j,k$, define 
\[
  \Omega_{i,j,k} := \{\lfloor g(Y_0) \rfloor = i,t_1 \in [\tau_j,\tau_{j^+}), t_2 \in [\tau_{j+k},\tau_{(j+k)^+})\}.
\]
By the triangle inequality, the law of total probability, Proposition \ref{prop3}, and Lemma \ref{lemiterate}, we have
\begin{multline}\label{conteq1}
 |\mathbb{E}(f_n(X_{t_2}))-\mathbb{E}(f_n(X_{t_1}))|
 \leq \sum_{i,j,k=0}^\infty \mathbb{E}\left(|f_n(X_{t_2})-f_n(X_{t_1})| \bigm| \Omega_{i,j,k} \right)\mathbb{P}(\Omega_{i,j,k})\\
 = \sum_{i,j,k=0}^\infty\mathbb{E}\left(|f_n(Y_{j+k})-f_n(Y_j)| \bigm| \Omega_{i,j,k} \right)\mathbb{P}(\Omega_{i,j,k})\\ 
 = \sum_{i,j=0}^\infty\sum_{k=1}^\infty\mathbb{E}\left(|f_n(Y_{j+k})-f_n(Y_j)| \bigm| \Omega_{i,j,k} \right)\mathbb{P}(\Omega_{i,j,k})\\
\leq \sum_{i,j=0}^\infty\sum_{k=1}^\infty  
 (2i^+ +(2j+k)R)\mathbb{P}(\Omega_{i,j,k}).
\end{multline}

If $k \geq 2$, note that
\begin{align*}
 \Omega_{i,j,k}
 &\subseteq \{\lfloor g(Y_0) \rfloor = i,t_1 \geq \tau_j, t_2-t_1 \geq \tau_{j+k}-\tau_{j^+}\}\\
 &\subseteq \{\lfloor g(Y_0) \rfloor = i\}
 \cap \left\{\theta t_1 \geq \sum_{\ell=1}^j \gamma_\ell\right\}
 \cap \left\{\theta (t_2-t_1) \geq \sum_{\ell=j+2}^{j+k} \gamma_\ell\right\},
\end{align*}
where sums with empty index ranges are taken to be zero.  The independence of $\{Y,\gamma_1,\gamma_2,\gamma_3,\ldots\}$ means that the three sets being intersected here are independent, so 
\[
 \mathbb{P}(\Omega_{i,j,k}) \leq 
\mathbb{P}\{\lfloor g(Y_0) \rfloor = i\}
\mathbb{P}\left\{\theta t_1 \geq \sum_{\ell=1}^j \gamma_\ell\right\}
\mathbb{P}\left\{\theta (t_2-t_1) \geq \sum_{\ell=j+2}^{j+k} \gamma_\ell\right\}.
\]
Bounding the last two probabilities using Lemma \ref{lemexpo} gives
\begin{equation}
 \label{p2}
 \mathbb{P}(\Omega_{i,j,k}) \leq 
\frac{(\theta t_1)^j}{j!}
\frac{(\theta (t_2 - t_1))^{k^-}}{(k^-)!}
\mathbb{P}\{\lfloor g(Y_0) \rfloor = i\}.
\end{equation}

For $k=1$, we need a more delicate estimate.  Note that
\begin{multline}
\mathbb{P}(\Omega_{i,j,1})
\leq
\mathbb{P}\{\lfloor g(Y_0) \rfloor = i\}
\mathbb{P}(t_1 \geq \tau_j \bigm| \lfloor g(Y_0) \rfloor = i)\\
\times \mathbb{P}( \tau_{j^+} \in (t_1,t_2] \bigm|
\lfloor g(Y_0) \rfloor = i,t_1 \geq \tau_j).
\label{p3}
\end{multline}
For $h$ in the finite set $(c(\mathsf{X}))^{[j^+]}$, define 
\begin{multline*}
 \Theta_{i,j,h} 
 := \{\lfloor g(Y_0) \rfloor = i,t_1 \geq \tau_j, c(Y_\cdot)|_{[j^+]} = h\}\\
 = \left\{\lfloor g(Y_0) \rfloor = i,t_1 \geq \sum_{\ell=1}^j \frac{\gamma_\ell}{h(\ell^-)},  c(Y_\cdot)|_{[j^+]} = h\right\}.
\end{multline*}
Then
\begin{multline}
 \mathbb{P}(\tau_{j^+} \in (t_1,t_2] \bigm|
\lfloor g(Y_0) \rfloor = i,t_1 \geq \tau_j)\\
= \sum_h \mathbb{P}(c(Y_\cdot)|_{[j^+]} = h \bigm| \lfloor g(Y_0) \rfloor = i,t_1 \geq \tau_j)\mathbb{P}(\tau_{j^+} \in (t_1,t_2] \bigm| \Theta_{i,j,h}),
\label{p4}
\end{multline}
where the summation is over $(c(\mathsf{X}))^{[j^+]}$.  Unraveling the formula for $\tau_{j^+}$ on the event $\Theta_{i,j,h}$, we have
\begin{multline}
 \mathbb{P}(\tau_{j^+} \in (t_1,t_2] \bigm| \Theta_{i,j,h})\\
 =
 \mathbb{P}\left( 
 h(j)\left(t_1 - \sum_{\ell=1}^j \frac{\gamma_{\ell}}{h(\ell^-)}\right)
  < \gamma_{j^+} \leq h(j)\left(t_2 - \sum_{\ell=1}^j \frac{\gamma_{\ell}}{h(\ell^-)}\right)
 \biggm| \Theta_{i,j,h}\right).
 \label{p5}
\end{multline}
Since $\gamma_{j^+}$ is independent of $(-h(j)\sum_{\ell=1}^j (\gamma_\ell/h(\ell^-)),\Theta_{i,j,h})$, Lemma \ref{explem} and \eqref{p5} imply that
\[
\mathbb{P}(\tau_{j^+} \in (t_1,t_2] \bigm| \Theta_{i,j,h})
 \leq h(j)(t_2-t_1) \leq \theta(t_2-t_1).
\]
Using this in \eqref{p4} gives 
\begin{multline*}
\mathbb{P}(\tau_{j^+} \in (t_1,t_2] \bigm|
\lfloor g(Y_0) \rfloor = i,t_1 \geq \tau_j)\\
\leq \theta(t_2-t_1) \sum_h \mathbb{P}(c(Y_\cdot)|_{[j^+]} = h \bigm| \lfloor g(Y_0) \rfloor = i,t_1 \geq \tau_j)
= \theta(t_2-t_1),  
\end{multline*}
and using that in \eqref{p3} gives
\begin{equation}
\label{p6}
\mathbb{P}(\Omega_{i,j,1})
\leq
\theta(t_2-t_1)
\mathbb{P}\{\lfloor g(Y_0) \rfloor = i\}
\mathbb{P}(t_1 \geq \tau_j \bigm| \lfloor g(Y_0) \rfloor = i).
\end{equation}
Similarly to the $k \geq 2$ case, but using independence this time, we can apply Lemma \ref{lemexpo} to deduce that
\begin{multline*}
 \mathbb{P}(t_1 \geq \tau_j \bigm| \lfloor g(Y_0) \rfloor = i)
\leq \mathbb{P}\left(\theta t_1 \geq \sum_{\ell=1}^j \gamma_\ell \biggm| \lfloor g(Y_0) \rfloor = i\right)\\
 = \mathbb{P}\left(\theta t_1 \geq \sum_{\ell=1}^j \gamma_\ell\right)
 \leq \frac{(\theta t_1)^j}{j!},
\end{multline*}
so \eqref{p6} implies that
\begin{equation}
\label{p7}
 \mathbb{P}(\Omega_{i,j,1})
\leq
 \frac{(\theta t_1)^j}{j!}
\theta(t_2-t_1)
\mathbb{P}\{\lfloor g(X_0) \rfloor = i\}.
\end{equation}

Applying \eqref{p2} and \eqref{p7} to \eqref{conteq1} yields the estimate
\begin{multline}
|\mathbb{E}(f_n(X_{t_2}))-\mathbb{E}(f_n(X_{t_1}))|\\
\leq
 \sum_{i,j,=0}^\infty \sum_{k=1}^\infty \mathbb{P}\{\lfloor g(X_0) \rfloor = i\}
 \frac{(\theta t_1)^j}{j!}
 \frac{(\theta(t_2-t_1))^{\max\{k^-,1\}}}{(\max\{k^-,1\})!}
(2i^+ +(2j+k)R).
\label{p8}
\end{multline}
Note that
\[
 \sum_{i=0}^\infty i\mathbb{P}\{\lfloor g(X_0) \rfloor = i\}
 = \mathbb{E}\left(\sum_{i=0}^\infty i\mathbf{1}_{\{\lfloor g(X_0)\rfloor = i\}}\right) 
 \leq \mathbb{E}(g(X_0)),
\]
which is finite by the hypothesis of integrability.  (See the proof of Proposition \ref{globalintegrability}.)  Applying this to the right-hand side of \eqref{p8} and calculating shows that this sum is the product of $t_2-t_1$ and an expression that is bounded when $t_1$ and $t_2$ are confined to a compact set.  This implies the continuity of $t \mapsto \mathbb{E}(f_n(X_t))$. 
\end{proof}

\subsection{Velocity}\label{velocity}
Our main result is the following theorem which shows that the expected velocity of of the centroid is dependent on the mean of the $\eta$ distribution and the rate at which the i-sites attach and detach.  In particular if the $\eta$ distribution is rotationally symmetric the velocity is zero.
\begin{theorem}\label{prop6} Let $\rho$ be a distribution on $\mathsf{X}$ such that $\sigma = \rho \circ \pi^{-1}$ and such that $f_i$ is $\rho$-integrable for every $i$.  Let $X$ be as in \cref{prop3}. Then 
\begin{equation}
 \frac{\partial}{\partial t} \mathbb{E}(f_n(X_t)) = \frac{\overline{\eta}\theta_d}{(\theta_d+\theta_a)^n}\left((\theta_d+\theta_a)^n-\theta_d^n\right)
\label{equ:velocity}
\end{equation}
for every $t>0$.
\end{theorem}

\begin{proof}
By Proposition \ref{continuityprop} and the mean value theorem for one-sided derivatives in \cite{Minassian:2007}, it suffices to verify that \eqref{equ:velocity} holds for the right-hand derivative.  Furthermore, the time-homogeneity of $X$ means that it suffices to do that verification for $t=0$.  (Proposition \ref{globalintegrability} implies that the hypothesized integrability condition translates to a corresponding integrability condition after a time shift, and Propositions \ref{fsprocess} and \ref{ltbfsprocess} imply that the projected distribution $\sigma$ does not change.)  

Let $Y$, $(\gamma_k)$, and $(\tau_k)$ be as in Proposition \ref{prop3}.
For each $(i,j) \in [n^+] \times [n^+]$, define $\Omega_{i,j} := \{\pi(Y_0)=i,\pi(Y_1)=j\}$.  The numerator of the relevant difference quotient is
\begin{equation}
 \mathbb{E}(f_n(X_t))-\mathbb{E}(f_n(X_0))
 = \sum_{i = 0}^n \sum_{j =0}^n \mathbb{E}(f_n(X_t)-f_n(X_0)\mid \Omega_{i,j})\mathbb{P}(\Omega_{i,j}).
 \label{ex}
\end{equation}
Recall that $(Y_0,Y_1)$ is $(\rho \times \mu)$-distributed, so
\begin{equation}
 \mathbb{P}(\Omega_{i,j})  = (\rho \times \mu)(\pi^{-1}(\{i\}) \times \pi^{-1}(\{j\})) = \int_{\pi^{-1}(\{i\})} \rho(d\mathsf{x}) \int_{\pi^{-1}(\{j\})} \mu(\mathsf{x},d\mathsf{y}).
 \label{p}
\end{equation}
By Lemma \ref{integrallem}, if $\mathsf{x} = (\psi,\mathbf{v})$, then
\begin{multline*}
 \int_{\pi^{-1}(\{j\})} \mu(\mathsf{x},d\mathsf{y}) 
 = \int_\mathsf{X} \mathbf{1}_{\pi^{-1}(\{j\})}(\mathsf{y}) \mu(\mathsf{x},d\mathsf{y})\\
 = \sum_{\ell \in \psi^{-1}(\{1\})} r_\ell(\psi)\mathbf{1}_{\pi^{-1}(\{j\})}(D_\ell(\mathsf{x})) + \sum_{\ell \in \psi^{-1}(\{0\})} r_\ell(\psi)\int_E \mathbf{1}_{\pi^{-1}(\{j\})}(A_\ell(\mathsf{x},\mathbf{x}))\,d\eta(\mathbf{x})\\
 = \sum_{\ell \in \psi^{-1}(\{1\})} r_\ell(\psi)\mathbf{1}_{\{j\}}(|s_\ell(\psi)|) + \sum_{\ell \in \psi^{-1}(\{0\})} r_\ell(\psi)\mathbf{1}_{\{j\}}(|s_\ell(\psi)|)\\
 = \sum_{\ell \in \psi^{-1}(\{1\})} r_\ell(\psi)\mathbf{1}_{\{j^+\}}(|\psi|) + \sum_{\ell \in \psi^{-1}(\{0\})} r_\ell(\psi)\mathbf{1}_{\{j^-\}}(|\psi|)\\
 = |\psi|\frac{\theta_d}{\theta_d|\psi|+\theta_a(n-|\psi|)}\mathbf{1}_{\{j^+\}}(|\psi|) + (n-|\psi|)\frac{\theta_a}{\theta_d|\psi|+\theta_a(n-|\psi|)}\mathbf{1}_{\{j^-\}}(|\psi|)\\
 = \frac{\theta_d(j^+)}{\theta_d(j^+)+\theta_a(n-j^+)}\mathbf{1}_{\pi^{-1}(\{j^+\})}(\mathsf{x}) + \frac{\theta_a(n-j^-)}{\theta_d(j^-)+\theta_a(n-j^-)}\mathbf{1}_{\pi^{-1}(\{j^-\})}(\mathsf{x}).
 \end{multline*}
Plugging this into \eqref{p}, and using the fact that $\rho \circ \pi^{-1} = \sigma$, we get
\begin{equation}
\mathbb{P}(\Omega_{i,j})
= \begin{cases}
 {\displaystyle \frac{\theta_d i}{\theta_d i+\theta_a(n-i)}}\sigma(\{i\}) & \text{if $i=j^+$,}\\
  {\displaystyle \frac{\theta_a (n-i)}{\theta_d i+\theta_a(n-i)}}\sigma(\{i\}) & \text{if $i=j^-$,}\\
  0 & \text{otherwise.}
\end{cases}
\label{pij}
\end{equation}

For the other factor of the summand in \eqref{ex}, we have
\begin{multline}
 \mathbb{E}(f_n(X_t)-f_n(X_0)\mid \Omega_{i,j}) \\
 = \sum_{k=0}^\infty \mathbb{E}(f_n(X_t)-f_n(X_0)\mid \Omega_{i,j} \cap \{t \in [\tau_k,\tau_{k^+})\})\mathbb{P}(t \in [\tau_k,\tau_{k^+}) \mid \Omega_{i,j}) \\
 = \sum_{k=0}^\infty \mathbb{E}(f_n(Y_k)-f_n(Y_0)\mid \Omega_{i,j} \cap \{t \in [\tau_k,\tau_{k^+})\})\mathbb{P}(t \in [\tau_k,\tau_{k^+}) \mid \Omega_{i,j}) \\
 = \mathbb{E}(f_n(Y_1)-f_n(Y_0)\mid \Omega_{i,j} \cap \{t \in [\tau_1,\tau_2)\})
\mathbb{P}(t \in [\tau_1,\tau_2) \mid \Omega_{i,j})\\
+ \sum_{k=2}^\infty \mathbb{E}(f_n(Y_k)-f_n(Y_0)\mid \Omega_{i,j} \cap \{t \in [\tau_k,\tau_{k^+})\})\mathbb{P}(t \in [\tau_k,\tau_{k^+}) \mid \Omega_{i,j}).
\label{ex2}
\end{multline}

We will compute the first term on the right in \eqref{ex2} and estimate the sum that follows it.  From the formula for $\tau_k$, the independence of $Y$ and $\{\gamma_k\}$, and the fact that $(Y_0,Y_1)$ is $(\rho \times \mu)$-distributed, we get
\begin{multline}
 \mathbb{E}(f_n(Y_1)-f_n(Y_0)\mid \Omega_{i,j} \cap \{t \in [\tau_1,\tau_2)\}) \\
 = \mathbb{E}\left(f_n(Y_1)-f_n(Y_0) \biggl| \Omega_{i,j} \cap \left\{ \frac{\gamma_1}{c(Y_0)} \leq t < \frac{\gamma_1}{c(Y_0)}+\frac{\gamma_2}{c(Y_1)}\right\}\right) \\
 = \mathbb{E}\left(f_n(Y_1)-f_n(Y_0) \biggl| \Omega_{i,j} \cap \left\{ \frac{\gamma_1}{\hat{c}(i)} \leq t < \frac{\gamma_1}{\hat{c}(i)}+\frac{\gamma_2}{\hat{c}(j)}\right\}\right) \\
 = \mathbb{E}(f_n(Y_1)-f_n(Y_0)\mid \Omega_{i,j})\\
= \frac{1}{\mathbb{P}(\Omega_{i,j})}\int_{\pi^{-1}(\{i\}) \times \pi^{-1}(\{j\})}(f_n(\mathsf{y})-f_n(\mathsf{x}))\,d(\rho \times \mu)(\mathsf{x},\mathsf{y})\\
= \frac{1}{\mathbb{P}(\Omega_{i,j})}\int_{\pi^{-1}(\{i\})} \rho(d\mathsf{x}) \int_{\pi^{-1}(\{j\})}(f_n(\mathsf{y})-f_n(\mathsf{x}))\mu(\mathsf{x},d\mathsf{y}).
\label{eq}
\end{multline}

Fix $\mathsf{x} = (\psi,\mathbf{v}) \in \mathsf{X}$, and note that Lemma \ref{integrallem} gives
\begin{multline}
 \int_{\pi^{-1}(\{j\})}(f_n(\mathsf{y})-   f_n(\mathsf{x}))\mu(\mathsf{x},d\mathsf{y})
= \int_\mathsf{X} \mathbf{1}_{\pi^{-1}(\{j\})}(\mathsf{y})(f_n(\mathsf{y})-f_n(\mathsf{x})) \mu(\mathsf{x},d\mathsf{y})\\
= \sum_{\ell \in \psi^{-1}(\{1\})} r_\ell(\psi)\mathbf{1}_{\pi^{-1}(\{j\})}(D_\ell(\mathsf{x}))(f_n(D_\ell(\mathsf{x}))-f_n(\mathsf{x}))\\
\qquad + \sum_{\ell \in \psi^{-1}(\{0\})}r_\ell(\psi)\int_E \mathbf{1}_{\pi^{-1}(\{j\})}(A_\ell(\mathsf{x},\mathbf{x}))(f_n(A_\ell(\mathsf{x},\mathbf{x}))-f_n(\mathsf{x}))\,d\eta(\mathbf{x}).
\label{p11}
\end{multline}
If $|\psi|\leq 1$, then $f_n(D_\ell(\mathsf{x}))=\mathbf{v}(n)=f_n(\mathsf{x})$, so the first sum on the right of \eqref{p11} vanishes.
If $|\psi| > 1$, then (because $\mathsf{x} \in \mathsf{X}$) that sum becomes
\begin{multline*}
 -\sum_{\ell \in \psi^{-1}(\{1\})} r_\ell(\psi)\mathbf{1}_{\{j\}}(|s_\ell(\psi)|)
 \frac{\mathbf{v}(\ell)-\mathbf{v}(n)}{|\psi|^-}\\
 = -\frac{\theta_d}{\theta_d|\psi|+\theta_a(n-|\psi|)}\frac{\mathbf{1}_{\{j^+\}}(|\psi|)}{|\psi|^-}\sum_{\ell \in \psi^{-1}(\{1\})} (\mathbf{v}(\ell)-\mathbf{v}(n)) = 0
\end{multline*}
again.  The second sum on the right of \eqref{p11} becomes
\begin{multline*}
\sum_{\ell \in \psi^{-1}(\{0\})}r_\ell(\psi)
\mathbf{1}_{\{j\}}(|s_\ell(\psi)|)
\int_E \frac{\mathbf{x}}{|\psi|^+}\,d\eta(\mathbf{x})
= \sum_{\ell \in \psi^{-1}(\{0\})} r_\ell(\psi)\mathbf{1}_{\{j^-\}}(|\psi|)\frac{\overline{\eta}}{|\psi|^+}\\
= (n-|\psi|)\frac{\theta_a}{\theta_d|\psi|+\theta_a(n-|\psi|)}\mathbf{1}_{\{j^-\}}(|\psi|)\frac{\overline{\eta}}{|\psi|^+}\\
 = \frac{\theta_a(n-\pi(\mathsf{x}))}{\theta_d\pi(\mathsf{x})+\theta_a(n-\pi(\mathsf{x}))}\mathbf{1}_{\pi^{-1}(\{j^-\})}(\mathsf{x})\frac{\overline{\eta}}{j}.
\end{multline*}
Thus, \eqref{p11} becomes
\begin{equation}
 \int_{\pi^{-1}(\{j\})} (f_n(\mathsf{y})-f_n(\mathsf{x}))\mu(\mathsf{x},d\mathsf{y}) 
 = \frac{\theta_a(n-\pi(\mathsf{x}))}{\theta_d\pi(\mathsf{x})+\theta_a(n-\pi(\mathsf{x}))}\mathbf{1}_{\pi^{-1}(\{j^-\})}(\mathsf{x})\frac{\overline{\eta}}{j}.
 \label{eq3}
\end{equation}

From \cref{eq3}, 
\begin{multline*}
 \int_{\pi^{-1}(\{i\})} \rho(d\mathsf{x}) \int_{\pi^{-1}(\{j\})}(f_n(\mathsf{y})-f_n(\mathsf{x}))\mu(\mathsf{x},d\mathsf{y}) \\
 = \int_{\pi^{-1}(\{i\})} \rho(d\mathsf{x}) \frac{\theta_a(n-\pi(\mathsf{x}))}{\theta_d\pi(\mathsf{x})+\theta_a(n-\pi(\mathsf{x}))}\mathbf{1}_{\pi^{-1}(\{j^-\})}(\mathsf{x})\frac{\overline{\eta}}{j}\\
 = \rho(\pi^{-1}(\{i\})) \frac{\theta_a(n-i)}{\theta_di+\theta_a(n-i)}\frac{\overline{\eta}}{i^+}
 = \sigma(\{i\}) \frac{\theta_a(n-i)}{\theta_d i+\theta_a(n-i)}\frac{\overline{\eta}}{i^+}
\end{multline*}
if $i=j^-$ and is $0$ otherwise.  Plugging this and \eqref{pij} into \eqref{eq} yields
\begin{equation}
\mathbb{E}(f_n(Y_1)-f_n(Y_0)\mid \Omega_{i,j} \cap\{t \in [\tau_1,\tau_2)\})
= \begin{cases}
 \frac{\overline{\eta}}{i^+} & \text{if $j=i^+$},\\
 0 & \text{if $j = i^-$}.
 \end{cases}
\label{eq4}
\end{equation}

From the independence of $Y$ and $\{\gamma_k\}$, along with the fact that the $\gamma_k$ are independent standard exponential random variables, we find that (for $j = i^\pm$)
\begin{multline}
 \mathbb{P}(t \in [\tau_1,\tau_2) \mid \Omega_{i,j})\\
 = \mathbb{P}\left( \frac{\gamma_1}{c(Y_0)} \leq t < \frac{\gamma_1}{c(Y_0)} + \frac{\gamma_2}{c(Y_1)} \mid \Omega_{i,j}\right) \\
 = \mathbb{P}\left( \frac{\gamma_1}{\hat{c}(i)} \leq t < \frac{\gamma_1}{\hat{c}(i)} + \frac{\gamma_2}{\hat{c}(j)} \mid \Omega_{i,j}\right) \\
 = \mathbb{P}\left( \frac{\gamma_1}{\hat{c}(i)} \leq t < \frac{\gamma_1}{\hat{c}(i)} + \frac{\gamma_2}{\hat{c}(j)}\right) 
 = \int_0^{t\hat{c}(i)}\int_{\hat{c}(j)(t-x/\hat{c}(i))}^\infty e^{-y}e^{-x}\,dy\,dx\\
 = \frac{e^{-t\hat{c}(i)}-e^{-t\hat{c}(j)}}{\hat{c}(j)-\hat{c}(i)}\hat{c}(i).
 \label{eq5}
\end{multline}

 From Lemma \ref{lemiterate}, and an argument similar to that in the proof of Proposition \ref{globalintegrability}, we see that
\begin{multline}
\left| \sum_{k=2}^{\infty} \mathbb{E}(f_n(Y_k) - f_n(Y_0) \bigm| \Omega_{i,j} \cap \{t \in [\tau_k, \tau_{k^+})\}) \mathbb{P}(t \in [\tau_k, \tau_{k^+}) \bigm| \Omega_{i,j} ) \right|  \\
 \leq \sum_{k=2}^{\infty} \mathbb{E}( \vert f_n(Y_k) - f_n(Y_0) \vert \bigm| \Omega_{i,j} \cap \{t \in [\tau_k, \tau_{k^+})\}) \mathbb{P}(t \in [\tau_k, \tau_{k^+}) \,|\, \Omega_{i,j} )  \\
 \leq \sum_{k=2}^{\infty} \mathbb{E}(2 g(Y_0) + kR \bigm| \Omega_{i,j} \cap \{ t \in [\tau_k, \tau_{k^+})\}) \mathbb{P}(t \in [\tau_k, \tau_{k^+}) \bigm| \Omega_{i,j} )  \\
 \leq 2 \mathbb{E}(g(Y_0) \bigm| \Omega_{i,j}) \sum_{k=2}^{\infty} \mathbb{P}( t \geq \tau_k \bigm| \Omega_{i,j})
 + R \sum_{k=2}^{\infty} k \p(t \geq \tau_k \bigm| \Omega_{i,j} ) \\
  \leq 2 \Ex(g(Y_0) \bigm| \Omega_{i,j}) (e^{t \theta} - 1 - t\theta) + Rt \theta (e^{t\theta} - 1). 
 \label{eq6} 
\end{multline}

\noindent Combining \eqref{ex2}, \eqref{eq4}, \eqref{eq5}, and \eqref{eq6} we have
\begin{multline}
\left| \mathbb{E}(f_n(X_t) - f_n(X_0) \bigm| \Omega_{i,i^+}) - \frac{\overline{\eta}}{i^+} \frac{e^{-t \hat{c} (i)} - e^{-t \hat{c}(i^+)}}{\hat{c}(i^+) - \hat{c}(i)} \hat{c}(i) \right|  \\
\leq 2 \Ex(g(Y_0) \bigm| \Omega_{i,i^+}) (e^{t \theta} - 1 - t\theta) + Rt \theta (e^{t\theta} - 1) 
\label{iplus} 
\end{multline}
and
\begin{multline}
|\mathbb{E}(f_n(X_t) - f_n(X_0) \bigm| \Omega_{i,i^-} )|  \\
  \leq 2 \mathbb{E}(g(Y_0) \bigm| \Omega_{i,i^-}) (e^{t \theta} - 1 - t\theta) + Rt \theta (e^{t\theta} - 1). 
\label{iminus} 
\end{multline}

\noindent Combining \eqref{ex}, \eqref{pij}, \eqref{iplus}, and \eqref{iminus} gives 
\begin{multline}
 \left| \mathbb{E}(f_n(X_t)) - \mathbb{E}(f_n(X_0)) - \sum_{i=0}^{n^-} \frac{ \overline{\eta}}{i^+} \frac{e^{-t \hat{c} (i)} - e^{-t \hat{c}(i^+)}}{\hat{c}(i^+) - \hat{c}(i)} \hat{c}(i) \frac{ \theta_a (n - i)}{ \theta_d i + \theta_a (n - i)} \sigma(\{i\}) \right|    \\
  \leq \sum_{i=0}^{n^-} (2 \mathbb{E}(g(Y_0) \bigm| \Omega_{i,i^+}) (e^{t \theta} - 1 - t\theta) + Rt \theta (e^{t\theta} - 1))   \left( \frac{ \theta_a (n - i)}{ \theta_d i + \theta_a (n - i)} \sigma(\{i\})  \right)    \\
 +  \sum_{i=1}^n (2 \mathbb{E}(g(Y_0) \bigm| \Omega_{i,i^-}) (e^{t \theta} - 1 - t\theta) + Rt \theta (e^{t\theta} - 1) ) \left( \frac{ \theta_d i }{\theta_d i + \theta_a (n - i)} \sigma(\{i\}) \right).   \label{eq28} 
\end{multline}

Since the summands on the right-hand side of \eqref{eq28} are quadratic in $t$, we just need to show that the sum on the left-hand side of \eqref{eq28} has the right limit when divided by $t$.  Calculating, we have
%
%
\begin{multline*}
 \lim_{t \searrow 0} \frac{1}{t}\sum_{i=0}^{n^-}
 \frac{\overline{\eta}}{i^+}
 \frac{e^{-t\hat{c}(i)}-e^{-t\hat{c}(i^+)}}{\hat{c}(i^+)-\hat{c}(i)}\hat{c}(i)
 \frac{\theta_a (n-i)}{\theta_d i+\theta_a(n-i)}\sigma(\{i\})\\
 = \sum_{i=0}^{n^-}
 \frac{\overline{\eta}}{i^+}
 \hat{c}(i)
 \frac{\theta_a (n-i)}{\theta_d i+\theta_a(n-i)}\sigma(\{i\})
 = \sum_{i=0}^{n^-}
 \frac{\overline{\eta}\theta_a (n-i)}{i^+}
 \frac{1}{(\theta_d+\theta_a)^n} \binom{n}{i}\theta_d^{n-i}\theta_a^i\\
 = \frac{\overline{\eta}\theta_d}{(\theta_d+\theta_a)^n}\sum_{i=0}^{n^-}
 \binom{n}{i^+}\theta_d^{n-i^+}\theta_a^{i^+}
 = \frac{\overline{\eta}\theta_d}{(\theta_d+\theta_a)^n}\sum_{j=1}^n
 \binom{n}{j}\theta_d^{n-j}\theta_a^j\\
 = \frac{\overline{\eta}\theta_d}{(\theta_d+\theta_a)^n}((\theta_d+\theta_a)^n-\theta_d^n).
\end{multline*}
\end{proof}

\section{Numerical results}\label{sec:num}
In this section, we compare numerical results for the average velocity of the CTCM and the theoretical results from Theorem~\ref{prop6}.  Additionally, we show numerical results for wait time distributions that are not exponential.

We simulate the CTCM by starting from an initial configuration with all the I-sites attached.  When the first I-site is detached, a new detach time is chosen from the specified distribution and the new centroid location is calculated.  The process evolves in a similar manner for all I-sites.  When an I-site attachment event occurs, we choose a random vector from the specified distribution and add this vector to the centroid location to determine the I-site attachment location.  Since the simulations do not start at a steady state distribution $\rho$, the simulations are allowed to evolve for 10 hours to allow the projection of the distribution to approach the steady-state distribution described in \cref{prop5}.  \cref{fig:NS1} shows the average $y$ velocity of the centroid location over a period of 65 hours for several values of $n$, the number of I-sites, (plotted as $\mathsf{X}$'s).  The solid line is the theoretical result from~\eqref{equ:velocity}.  Three different sets of simulations are shown, all have exponential wait times for both the attach and detach time of the I-sites.  They differ in the values for $\theta_d$ (the detachment rate).  As the figure shows, the numerical simulations agree with the theoretical results.  
\begin{figure}[h]
  \centerline{\includegraphics*[width=.5\textwidth]{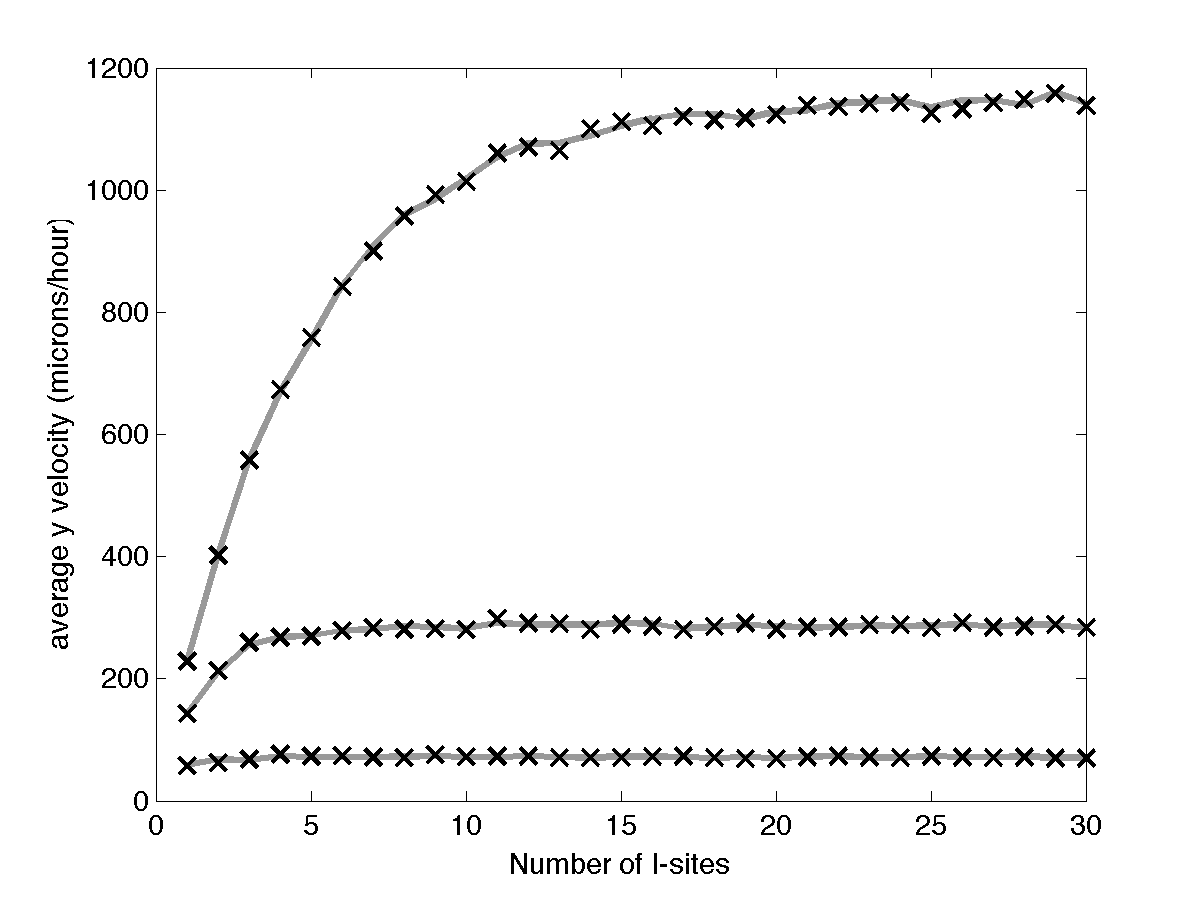}}
\caption{This figure compares the numerical simulations for the
  CTCM with
   the theoretical results (\ref{equ:velocity}).  The `$\mathsf{x}$' denote the numerical
   simulations and the line shows the theoretical results. The
   horizontal axis indicates the number of I-sites for the cell and
   the vertical axis shows the average velocity of the 
   $y$ position of the centroid over a 65 hour period (results are similar for
    the $x$ direction).  The three different sets of `$\mathsf{x}$' and lines in decreasing order are for $\theta_d=\frac{1}{5}$, $\theta_d=\frac{1}{20}$, and $\theta_d=\frac{1}{80}$.  In all the simulations $\theta_a =\frac{1}{20}$; the units are per second.}
\label{fig:NS1}
\end{figure}

For biologically realistic situations, the duration of the adhesion sites may not be exponentially distributed.  Adhesion sites are reinforced or degraded in a manner which seems dependent on the forces applied to the adhesion sites and the amount of time they have been attached.  All this indicates that the attachment/detachment process is complicated.  In one study on dictyostelium discoideum the attachment duration of actin foci was measured to be a "continuous Poisson distribution" \cite{Uchida:2004:DNF}. (The continuous Poisson distribution is a distribution which when rounded to the integers is a Poisson distribution \cite{Marsaglia:1986:IGF}.)  In Figure~\ref{fig:NS2} we compare two simulations where the attachment/detachment processes are not Markovian as indicated by the different distributions for the wait times.  In the figure, simulations with a continuous Poisson, an exponential, and a normal distribution for wait times are shown.  One can see that the non-Markov processes do not follow the same pattern as the Markov process.  Yet, the two non-Markov processes do seem to follow some curve.  In future work we plan to investigate the effect of dropping the Markov property, for a more realistic scenario.  
\begin{figure}[h]
  \centerline{\includegraphics*[width=.5\textwidth]{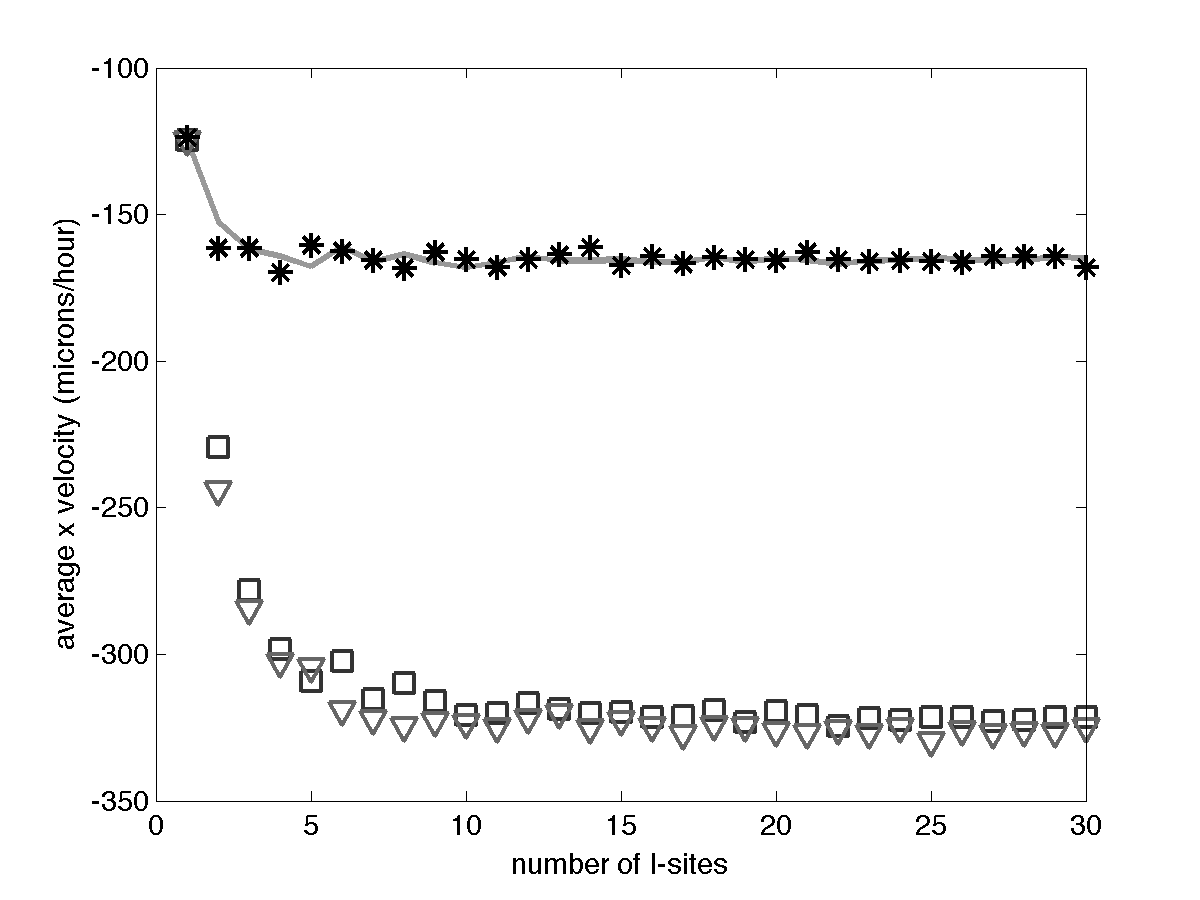}}
  \caption{This figure shows a comparison between the numerical simulations for the
    CTCM with different wait time distributions
    and the theoretical results.
    The symbols indicate the numerical simulations and the line shows the
    theoretical results.  The $\ast$ denote simulations with
    exponential distributions, the $\triangledown$ denote simulations with normal distributions (truncated to prevent negative times) with
    deviation 1, and the $\Box$ denote simulations with "continuous Poisson distributions". The mean time to detachment is 60 seconds and the mean time to attachment is 20 seconds for all the simulations.  The
    horizontal axis indicates the number of I-sites and
    the vertical axis shows the average velocity in the $x$ direction.}
\label{fig:NS2}
\end{figure}
\section{Discussion}\label{sec:discussion}

Although the differential-equation model considered was motivated by the motion of a cell, the results may be relevant to any system of particles or objects undergoing centrally-controlled motion. In this paper we considered the CTCM, a model that paralleled the differential-equation model by effectively setting the intracellular forces equal to infinity. Without assuming the typical white noise hypothesis, we were able to predict the expected velocity of the CTCM using the theory of pure jump-type Markov processes and Markov chains. The key result is Theorem~\ref{prop6}, which gave a formula that predicts the time rate of change of the expected position of the cell. 
This formula will be used to prove that the average velocity of a cell, as predicted by the ordinary differential equation model is dependent on the adhesion dynamics and not the cell force.  The formula gives an indication of this by showing the dependence of the centroid on $\overline{\eta}$ the mean of the perturbation of the adhesion sites, the number of adhesion sites, and the expected attach and detach times.  Of course, the assumption that the adhesion dynamics are Markovian is a simplifying assumption and will need to be modified in the future.  The results may also be considered as a step toward general understanding of differential equation models with randomness. 

\section*{Acknowledgement}
The authors wish to thank the anonymous referees for their helpful feedback and suggestions.
\section{Bibliography}
\bibliographystyle{elsarticle-num}
\bibliography{NSFstochasticgrant}

\begin{thebibliography}{10}
\expandafter\ifx\csname url\endcsname\relax
  \def\url#1{\texttt{#1}}\fi
\expandafter\ifx\csname urlprefix\endcsname\relax\def\urlprefix{URL }\fi
\expandafter\ifx\csname href\endcsname\relax
  \def\href#1#2{#2} \def\path#1{#1}\fi

\bibitem{Krawczyk:1971:PEC}
W.~Krawczyk, A pattern of epidermal cell migration during wound healing, The
  Journal of cell biology 49~(2) (1971) 247--263.

\bibitem{Tanner:2009:CMC}
K.~Tanner, D.~Ferris, L.~Lanzano, B.~Mandefro, W.~Mantulin, D.~Gardiner,
  E.~Rugg, E.~Gratton, Coherent movement of cell layers during wound healing by
  image correlation spectroscopy, Biophysical journal 97~(7) (2009) 2098--2106.

\bibitem{Yilmaz:2010:MMM}
M.~Yilmaz, G.~Christofori, Mechanisms of motility in metastasizing cells,
  Molecular Cancer Research 8~(5) (2010) 629--642.

\bibitem{Keller:2000:MCE}
R.~Keller, L.~Davidson, A.~Edlund, T.~Elul, M.~Ezin, D.~Shook, P.~Skoglund,
  Mechanisms of convergence and extension by cell intercalation, Philosophical
  Transactions of the Royal Society of London. Series B: Biological Sciences
  355~(1399) (2000) 897--922.

\bibitem{Mammoto:2010:MCT}
T.~Mammoto, D.~Ingber, Mechanical control of tissue and organ development,
  Development 137~(9) (2010) 1407--1420.

\bibitem{Rieu:2009:MDS}
J.-P. Rieu, T.~Saito, H.~Delano{\"e}-Ayari, Y.~Sawada, R.~R. Kay, Migration of
  dictyostelium slugs: anterior-like cells may provide the motive force for the
  prespore zone, Cell Motil Cytoskeleton 66~(12) (2009) 1073--86.
\newblock \href {http://dx.doi.org/10.1002/cm.20411}
  {\path{doi:10.1002/cm.20411}}.

\bibitem{Dallon:2013:FBM}
J.~C. Dallon, M.~Scott, W.~V. Smith, A force based model of individual cell
  migration with discrete attachment sites and random switching terms, Journal
  of Biomechanical Engineering 135~(7) (2013) 071008--071008--10.
\newblock \href {http://dx.doi.org/10.1115/1.4023987}
  {\path{doi:10.1115/1.4023987}}.

\bibitem{Dallon:2013:CSI}
J.~C. Dallon, E.~J. Evans, C.~P. Grant, W.~V. Smith,
  \href{http://dx.doi.org/10.1016/j.mbs.2013.09.005}{Cell speed is independent
  of force in a mathematical model of amoeboidal cell motion with random
  switching terms}, Math. Biosci. 246~(1) (2013) 1--7.
\newblock \href {http://dx.doi.org/10.1016/j.mbs.2013.09.005}
  {\path{doi:10.1016/j.mbs.2013.09.005}}.
\newline\urlprefix\url{http://dx.doi.org/10.1016/j.mbs.2013.09.005}

\bibitem{Kallenberg}
O.~Kallenberg, \href{http://opac.inria.fr/record=b1098179}{Foundations of
  modern probability}, Probability and its applications, Springer, New York,
  2002.
\newline\urlprefix\url{http://opac.inria.fr/record=b1098179}

\bibitem{Friedl:2009:CCM}
P.~Friedl, D.~Gilmour, Collective cell migration in morphogenesis, regeneration
  and cancer, Nature Reviews Molecular Cell Biology 10~(7) (2009) 445--457.

\bibitem{Ulrich:2009:TCM}
F.~Ulrich, C.-P. Heisenberg, Trafficking and cell migration, Traffic 10~(7)
  (2009) 811--8.
\newblock \href {http://dx.doi.org/10.1111/j.1600-0854.2009.00929.x}
  {\path{doi:10.1111/j.1600-0854.2009.00929.x}}.

\bibitem{Gumbiner:1996:CAM}
B.~M. Gumbiner, Cell adhesion: the molecular basis of tissue architecture and
  morphogenesis, Cell 84~(3) (1996) 345--57.

\bibitem{Dallon:2004:HCM}
J.~C. Dallon, H.~G. Othmer, How cellular movement determines the collective
  force generated by the {{\em Dictyostelium discoideum}} slug 231 (2004)
  203--222.

\bibitem{Cinlar}
E.~C{\i}nlar, \href{http://dx.doi.org/10.1007/978-0-387-87859-1}{Probability
  and stochastics}, Vol. 261 of Graduate Texts in Mathematics, Springer, New
  York, 2011.
\newblock \href {http://dx.doi.org/10.1007/978-0-387-87859-1}
  {\path{doi:10.1007/978-0-387-87859-1}}.
\newline\urlprefix\url{http://dx.doi.org/10.1007/978-0-387-87859-1}

\bibitem{MeynTweedie}
S.~Meyn, R.~L. Tweedie,
  \href{http://dx.doi.org/10.1017/CBO9780511626630}{Markov chains and
  stochastic stability}, 2nd Edition, Cambridge University Press, Cambridge,
  2009.
\newblock \href {http://dx.doi.org/10.1017/CBO9780511626630}
  {\path{doi:10.1017/CBO9780511626630}}.
\newline\urlprefix\url{http://dx.doi.org/10.1017/CBO9780511626630}

\bibitem{Ross}
S.~Ross, A first course in probability, ninth Edition, Pearson, Boston, 2004.

\bibitem{Minassian:2007}
D.~Minassian, A mean value theorem for one-sided derivatives, Amer. Math.
  Monthly 114~(1) (2007) 28.

\bibitem{Uchida:2004:DNF}
K.~Uchida, S.~Yumura, Dynamics of novel feet of dictyostelium cells during
  migration, Journal of Cell Science 117~(8) (2004) 1443--1455.
\newblock \href {http://dx.doi.org/DOI 10.1242/jcs.01015} {\path{doi:DOI
  10.1242/jcs.01015}}.

\bibitem{Marsaglia:1986:IGF}
G.~Marsaglia, The incomplete $\gamma$ function as a continuous poisson
  distribution, Computers \& Mathematics with Applications 12~(5) (1986)
  1187--1190.

\end{thebibliography}

\end{document}